\documentclass[smallextended]{svjour3} 

\usepackage{etoolbox}
\newtoggle{in_office}

\toggletrue{in_office}
\togglefalse{in_office}

\usepackage{defs}
\usepackage{prjohns2MacrosTA}

\allowdisplaybreaks

\smartqed 

\usepackage{url}
\usepackage{amssymb}
\usepackage{mathtools}
\usepackage[ruled,lined]{algorithm2e}
\usepackage{amsmath}
\usepackage{enumerate}\usepackage{epstopdf}

\usepackage{fixltx2e}
\usepackage{marvosym}
\newcommand{\envelope}{(\raisebox{-.5pt}{\scalebox{1.45}{\Letter}}\kern-1.7pt \hspace{0.7mm})}

\begin{document}


\title{Faster Subgradient Methods for Functions with H\"olderian Growth}



\author{Patrick R. Johnstone\and  Pierre Moulin}

\institute{
Patrick R. Johnstone \at 
Department of Management Sciences and Information Systems, Rutgers Business School Newark and New Brunswick, Rutgers University\\
\email{patrick.r.johnstone@gmail.com}
\and 
Pierre Moulin \at 
Coordinated Science Laboratory, University of Illinois,Urbana, IL 61801, USA\\
\email{pmoulin@illinois.edu} 
}

\date{Received: date / Accepted: date}

\maketitle
\begin{abstract}
The purpose of this manuscript is to derive new convergence results for several subgradient methods applied to minimizing nonsmooth convex functions with H\"olderian growth. The growth condition is satisfied in many applications and includes functions with quadratic growth and weakly sharp minima as special cases. To this end there are three main contributions. First, for a constant and sufficiently small stepsize, we show that the subgradient method achieves linear convergence up to a certain region including the optimal set, with error of the order of the stepsize. Second, if appropriate problem parameters are known, we derive a decaying stepsize which obtains a much faster convergence rate than is suggested by the classical $O(1/\sqrt{k})$ result for the subgradient method. Thirdly we develop a novel ``descending stairs" stepsize which obtains this faster convergence rate and also obtains linear convergence for the special case of weakly sharp functions.
We also develop an adaptive variant of the ``descending stairs" stepsize which achieves the same convergence rate without requiring an error bound constant which is difficult to estimate in practice.

\end{abstract}

\section{Introduction}
\subsection{Motivation and Background}
In this manuscript we consider the following problem:
\begin{eqnarray}\label{Prob1}
\min_{x\in\calC} h(x),
\end{eqnarray}
where $\calC$ is a convex, closed, and nonempty subset of a real Hilbert space $\calH$, and $h:\calH\to\mathbb{R}$ is a convex and closed function. We do not assume $h$ is smooth or strongly convex. Problem (\ref{Prob1}) arises in many applications such as image processing, machine learning, compressed sensing, statistics, and computer vision \cite{hastie2009elements,agrb1992maximum,yang2015rsg}.

We focus on the class of \emph{subgradient methods} for solving this problem, which were first studied in the 1960s \cite{shor2012minimization,goffin1977convergence}. Since then, these methods have been used extensively because of their simplicity and low per-iteration complexity \cite{shor2012minimization,goffin1977convergence,rosenberg1988geometrically,nedic2010effect,nedic2001convergence,nemirovski2009robust}. Such methods only evaluate a subgradient of the function at each iteration. 
However in general they have a slow worst-case convergence rate of $h(\hat{x}_k)-\min_{x\in\calC}h(x)\leq O(1/\sqrt{k})$ after $k$ subgradient evaluations for a particular averaged point $\hat{x}_k$. In this manuscript we show how a structural assumption for Problem (\ref{Prob1}) that is commonly satisfied in practice yields faster subgradient methods.

The structural assumption we consider is the \emph{H\"older error bound} (throughout referred to as either HEB, HEB$(c,\theta)$, or H\"olderian growth). We assume that $h$ satisfies
\begin{eqnarray*}\label{errorBound}
\hspace{3cm}h(x)-h^*\geq c d(x,\calX_h)^{\frac{1}{\theta}},\quad\forall x\in\mathcal{C},\hspace{2.5cm}(\text{HEB})
\end{eqnarray*}
where
\begin{itemize}
	\item $\theta\in(0,1]$ is the ``error bound exponent",
	\item $c>0$ is the ``error bound constant",
	\item $h^*=\min_{x\in\calC}h(x)$ is the optimal value,
\item 
$\calX_h\triangleq\{x\in\calC:h(x)=h^*\}$ is the solution set (assumed to be nonempty), and
\item 
 $d(x,\calX_h)=\inf_{x^*\in\calX_h}\|x-x^*\|$. 
\end{itemize}
In general, an ``error bound" is an upper bound on the distance of a point to the optimal set by some residual function. The study of error bounds has a long tradition in optimization, sensitivity analysis, systems of inequalities, projection methods, and convergence rate estimation \cite{li2013global,tseng2010approximation,zhou2015unified,xu2016accelerate,bolte2015error,burke1993weak,zhang2013gradient,pang1997error,luo1993error,ferris1991finite,burke2002weak,karimi2016linear,beck2015linearly} In recent years there has been much renewed interest in the topic.  HEB is often referred to as the \emph{{\L}ojaziewicz error bound} \cite{bolte2007lojasiewicz} and is also related to the \emph{Kurdyka--\L ojaziewicz (KL) inequality} \cite{bolte2015error}. In fact in \cite{bolte2015error} it was shown that the KL inequality is equivalent to HEB for convex, closed, and proper functions.

There are three main motivations for studying the behavior of algorithms for problems satisfying  HEB. Firstly HEB holds for problems arising in many applications. In fact for a semialgebraic function, HEB is guaranteed to hold on a compact set for some $\theta$ and $c$ \cite{bolte2015error}. Secondly, many algorithms have been shown to achieve significantly faster convergence behavior when HEB is satisfied. Thirdly, under HEB it has been possible to develop even faster methods. 

The two most common instances of HEB in practice are $\theta=1/2$ and $\theta=1$. The $\theta=1/2$ case is often referred to as the \emph{quadratic growth condition} (QG) \cite{karimi2016linear}. The $\theta=1$ case is often referred to by saying the function has \emph{weakly sharp minima} (WS) \cite{burke2002weak}. The function itself may also be called a weakly sharp function. There are also a small number of applications where $\theta\neq 1/2$ or $1$, such as $L_d$ regression with $d\neq 1,2$ \cite{agrb1992maximum}. 
 
Due to its prevalence in applications, many recent papers have studied QG (the $\theta=1/2$ case). QG has been used to show a \emph{linear} convergence rate of the objective function values for various algorithms, such as the proximal gradient method, that would otherwise only guarantee sublinear convergence \cite{zhang2016new,beck2015linearly,zhou2015unified,karimi2016linear}. 
Many papers have discovered connections between QG and other error bounds and conditions known in the literature. Most importantly it was shown in \cite[Appendix A]{karimi2016linear} that for convex functions, QG is equivalent to the \emph{Luo-Tseng} error bound \cite{luo1993error}, the \emph{Polyak-{\L}ojaziewicz} condition \cite{karimi2016linear}, and the \emph{restricted secant inequality} \cite{zhang2013gradient}. 

Weakly sharp functions (i.e. HEB with $\theta=1$) have been studied in many papers, for example \cite{burke2002weak,ferris1991finite,pang1997error,shor2012minimization,nedic2010effect,poljak1978nonlinear,yang2015rsg,supittayapornpong2016staggered,attouch2013convergence}. For such functions \cite{ferris1991finite} showed that the proximal point method converges to a minimum in a \emph{finite} number of iterations. This is interesting because this method would otherwise only have an $O(1/k)$ rate. 

\subsection{Our Contributions}
Recall the definition of the subgradient of $h$ at $x$ \cite[Def. 16.1]{bauschke2011convex}:
\begin{eqnarray*}
\partial h(x) \triangleq \{g\in\mathcal{H}:h(y)\geq h(x)+\langle g,y-x\rangle, \forall y\in\mathcal{H}\}.
\label{eq:ineq3}
\end{eqnarray*}
Define the \emph{subgradient method} as 
\begin{eqnarray}\label{iterSG}
 x_{k+1}=P_{\calC}(x_k-\alpha_k g_k):\quad\forall k\geq 1, g_k\in\partial h(x_k),\,\, x_1\in \calC,
 \end{eqnarray}
where $P_{\calC}$ denotes the projection onto $\calC$ and the choice of the \emph{stepsize} $\alpha_k>0$ is left unspecified. 
 Despite the long history of analysis of subgradient methods, the simplest stepsize choices for (\ref{iterSG}) have not been studied for objective functions satisfying HEB. These are the constant stepsize, $\alpha_k=\alpha$, and the decaying stepsize, $\alpha_k=\alpha_1 k^{-p}$ for $p>0$. This brings us to our contributions in this manuscript. 

Firstly we determine the convergence rate of a constant stepsize choice which previously had only been determined for the special case of $\theta=1/2$ (see \cite[Prop. 2.4]{nedic2001convergence}). Interestingly, \emph{for any} $\theta\in(0,1]$ the method obtains a linear convergence rate for $d(x_k,\calX_h)$, up to a specific tolerance level of order $O(\alpha^{\theta})$. 

Secondly, we derive decaying stepsizes which obtain much faster rates than the classical subgradient method if appropriate problem parameters are available.
The classical analysis of the subgradient method leads to the rate
\begin{eqnarray*}
h(\hat{x}_k)-h^*\leq O(k^{-\frac{1}{2}}),
\end{eqnarray*}
where $\hat{x}_k$ is a specific average of the previous iterates and $\alpha_k=O(1/\sqrt{k})$ \cite{nemirovski2009robust}. Combining this with HEB yields
\begin{eqnarray*}\label{ConvClassic}
d(\hat{x}_k,\calX_h)\leq O(k^{-\frac{\theta}{2}}).
\end{eqnarray*}
This rate is slower than the result of our specialized analysis. We show that with stepsize $\alpha_k=\alpha_1 k^{-p}$ and the proper choice of $p$ and $\alpha_1$, the subgradient method can obtain the convergence rate
\begin{eqnarray}\label{optDecay}
d(x_k,\calX_h)\leq O(k^{-\frac{\theta}{2(1-\theta)}}),\quad\forall \theta<1.
\end{eqnarray}
It can be seen that the absolute value of the exponent is a factor $1/(1-\theta)$ larger in our analysis. 

Our third major contribution is  a new ``descending stairs" stepsize choice for the subgradient method (DS-SG). The method achieves the convergence rate given in (\ref{optDecay}) for $\theta<1$. In addition, for the case $\theta=1$ it achieves linear convergence. Unlike the methods of \cite{renegar2015framework,renegar2016efficient} and \cite[Exercise 6.3.3]{bertsekas1999nonlinear}, which also obtain linear convergence when $\theta=1$, our proposal does not require knowledge of $h^*$. The methods of \cite{supittayapornpong2016staggered,shor2012minimization,goffin1977convergence} have a similar complexity for $\theta=1$ but cannot handle $\theta<1$. The Restarted Subgradient method (RSG) \cite{yang2015rsg} obtains the same iteration complexity but requires averaging which is disadvantageous in applications where the solution is sparse (or low rank) because it can spoil this property \cite{davis2017three}. (In Section \ref{Discuss} we discuss other problems with averaging.) An advantage of RSG is it only requires that HEB be satisfied locally, i.e. on a sufficiently-large level set of $h$. However in the important case where $\theta=1$ this makes no difference, because if HEB holds with $\theta=1$ on any compact set, then it holds globally \cite{burke1993weak}. Furthermore for many applications with $\theta<1$, HEB is satisfied globally \cite{bolte2015error,karimi2016linear}.   


DS-SG, RSG, and several other methods \cite{goffin1977convergence,shor2012minimization} require knowledge of the constant $c$ in HEB which can be hard to estimate in practice. This motivates us to develop our final major contribution: a ``doubling trick" for DS-SG which automatically adapts to the unknown error bound constant and still obtains the same iteration complexity, up to a small constant. We call this method  the ``doubling trick descending stairs subgradient method" (DS2-SG).  The competing methods of \cite{yang2015rsg,supittayapornpong2016staggered,shor2012minimization,goffin1977convergence} all require knowledge of $c$. The authors of \cite{yang2015rsg} proposed an adaptive method which does not require $c$, however it only works for $\theta<1$. 

In summary, our contributions under HEB are as follows: 
\begin{enumerate}
\item We show that the subgradient method with a constant stepsize obtains linear convergence for $d(x_k,\calX_h)$ to within a region of the optimal set for all $\theta\in(0,1]$. 
\item We derive a decaying stepsize with faster convergence rate than the classical subgradient method.
\item We develop a new ``Descending Stairs" stepsize with iteration complexity $O(\epsilon^{1-\frac{1}{\theta}})$ when $\theta<1$ and $\ln\frac{1}{\epsilon}$ when $\theta=1$ for finding a point such that $d(x_k,\calX_h)^2\leq\epsilon$. We also develop an adaptive variant which does not need $c$ but retains the same iteration complexity up to a small constant. 
\end{enumerate}
 \begin{table}
  	\caption{Summary of our contributions for constant, decaying (polynomial), DS-SG, and DS2-SG stepsizes. The given convergence rates are for $d(x_k,\calX_h)^2$. We list the cases $\theta=1$ and $\theta<1$ seperately. Goffin \cite{goffin1977convergence} developed a geometrically decaying stepsize which obtains geometric convergence rate for the case $\theta=1$ with known $c$ (see also \cite[Sec. 2.3]{shor2012minimization}.}\label{tableux}
  	\begin{tabular}{|c||c|c|c|c|}
  		\hline
    &constant& decaying & DS-SG & DS2-SG\\		
    \hline
  	$\theta=1$& $q^k+O(\alpha^{2\theta})$ & $O(q^k)$, Goffin \cite{goffin1977convergence}&$O(q^k)$&$O(q^k)$, $c$ not required\\[2ex]
  	\hline 
  	$\theta<1$ & $q^k+O(\alpha^{2\theta})$ & $O\left(k^{\frac{\theta}{\theta-1}}\right)$ 
  	&$O\left(k^{\frac{\theta}{\theta-1}}\right)$ & $O\left(k^{\frac{\theta}{\theta-1}}\right)$, $c$ not required
  	\\[2ex]
  	\hline
  	\end{tabular}
  	
  \end{table}
  Our contributions are summarized in Table \ref{tableux}.
  
The outline for the manuscript is as follows. In Sec. \ref{Sec_subgHist} we discuss some previously known results for subgradient methods applied to functions satisfying HEB. In Sec. \ref{secKeyRec} we derive the key recursion which describes the subgradient method under HEB and allows us to obtain convergence rates. In Sec. \ref{sec_const} we determine the behavior of a constant stepsize. In Sec. \ref{sec_itercomp} we derive a constant stepsize with explicit iteration complexity. In Sec. \ref{secRest} we develop our proposed DS-SG. In Sec. \ref{secAdapt} we develop the variant, DS2-SG, which does not require the error bound constant. In Sec. \ref{sec_sum} we derive a decaying stepsize with faster convergence rate than the classical decaying stepsize. In Sec. \ref{Sec_decay}, we derive convergence rates under HEB for some classical, decaying, and nonsummable stepsizes. These results are proved in Sec. \ref{secProofDecay}. Finally, Sec. \ref{sec_numerical} features numerical experiments to test some of the theoretical findings of this paper. 

\section{Prior Work on Subgradient Methods under HEB}\label{Sec_subgHist}

There were a few early works that studied the subgradient method under conditions related to HEB with $\theta=1$. In \cite[Thm 2.7, Sec. 2.3]{shor2012minimization}, Shor proposed a geometrically decaying stepsize which obtains a linear convergence rate under a condition equivalent to HEB with $\theta=1$. The stepsize depends on explicit knowledge of the error bound constant $c$, a bound on the subgradients, and the initial distance $d(x_1,\calX_h)$. Goffin \cite{goffin1977convergence} extended the analysis of \cite{shor2012minimization} to a slightly more general notion than HEB.\footnote{Our analysis also holds for Goffin's condition.}  Note that our optimal decaying stepsize, derived in Sec. \ref{sec_sum}, is a natural extension of Goffin's geometrically-decaying stepsize to $\theta<1$.  Rosenburg \cite{rosenberg1988geometrically} extended Goffin's results to constrained problems. In \cite{poljak1978nonlinear}, Polyak showed that Goffin's  method still converges linearly when the subgradients are corrupted by bounded, deterministic noise.

The paper \cite{nedic2010effect} also considers functions satisfying HEB with $\theta=1$ with (deterministically) noisy subgradients. For constant stepsizes, they show convergence of $\lim\inf h(x_k)$ to $h^*$ plus a tolerance level depending on noise. For diminishing stepsizes, they show that $\lim\inf h(x_k)$ actually converges to $h^*$ despite the noise. However \cite{nedic2010effect} does not discuss \emph{convergence rates}, which is the topic of our work.

As mentioned in the introduction, \cite{yang2015rsg} introduced the \emph{restarted subgradient method} (RSG) for when $h$ satisfies HEB. The method implements a predetermined number of averaged subgradient iterations with a constant stepsize and then restarts the averaging and uses a new, smaller stepsize. The authors show that after $O(\epsilon^{2(\theta-1)}\log\frac{1}{\epsilon})$ iterations the method is guaranteed to find a point such that $h(x_k)-h^*\leq\epsilon$. For $\theta=1$ this is a logarithmic iteration complexity. This improves the iteration complexity of the classical subgradient method which is $O(\epsilon^{-2})$. Differences between our results and RSG will be discussed in Sec. \ref{secDD_discuss}. 

The recent paper \cite{xu2016accelerate} extends RSG to stochastic optimization. In particular they provide a similar restart scheme that can also handle stochastic subgradient calls, and guarantees $h(x)-h^*\leq\epsilon$ with high probability. The iteration complexity is the same as for RSG, up to a constant. However, this constant is large leading to a large number of inner iterations, making it potentially difficult to implement the method in practice. 

For WS functions, the paper \cite{supittayapornpong2016staggered} introduced a method similar to RSG except it does not require averaging at the end of each constant stepsize phase. The method also obtains a logarithmic iteration complexity in the $\theta=1$ case.  This method is essentially a special case of our proposed DS-SG for $\theta=1$. 

The paper \cite{gilpin2012first} is concerned with a two-person zero-sum game equilibrium problem with a linear payoff structure. The authors show that finding the solution to the equilibrium problem is equivalent to a WS minimization problem. Using this fact, they derive a method based on Nesterov's smoothing technique with logarithmic iteration complexity. This is superior to the $O(1/\epsilon)$ of standard Nesterov smoothing. Connections between our results and \cite{gilpin2012first} are discussed in Section \ref{secDD_discuss}.  

The work \cite{lim2011convergence} studies stochastic subgradient descent under the assumption that the function satisfies WS locally and QG globally. They show a faster convergence rate of the iterates to a minimizer, both in expectation and with high probability, than is known under the classical analysis. 

The work \cite{freund2015new} proposes a new subgradient method for functions satisfying a similar condition to HEB but with $h^*$ replaced by a strict lower bound on $h^*$. Like RSG, this algorithm has a logarithmic dependence on the initial distance to the solution set. However it still obtains an $O(1/\epsilon^2)$ iteration complexity, which is the same as the classical subgradient method.

In \cite{renegar2015framework,renegar2016efficient} Renegar presented a framework for converting a convex conic program to a general convex problem with an affine constraint, to which projected subgradient methods can be applied. He further showed how this can be applied to general convex optimization problems, such as Prob. (\ref{Prob1}), by representing them as a conic problem. For the special case where the objective and constraint set is polyhedral, one of the subgradient methods proposed by Renegar has a logarithmic iteration complexity \cite[Cor. 3.4]{renegar2015framework}. The main drawback of this method is that it requires knowledge of the optimal value, $h^*$. It also requires a point in the interior of the constraint set. Similarly the stepsizes proposed in Thm. 2 of \cite[Sec 5.3.]{PolyakIntro} and \cite[Prop. 2.11]{nedic2001convergence} depend on exact knowledge of $h^*$ and also obtain a logarithmic iteration complexity under WS. 

The work \cite{noll2014convergence} explores subgradient-type algorithms for nonsmooth nonconvex functions satisfying the KL inequality. A procedure was developed for selecting a subgradient at each iteration which results in a decrease in objective value, thereby leading to convergence to a critical point. The selection procedure typically involves either storing a collection of past subgradients and solving a convex program, or suitably backtracking the stepsize until a certain condition is met.

For WS functions, it is known that there are subgradient methods which obtain linear convergence \cite{goffin1977convergence,shor2012minimization,yang2015rsg}. A different assumption, known as partial smoothness, has been used to show \emph{local} linear convergence of proximal gradient methods \cite{hare2004identifying,liang2017activity}. We mention that the partial smoothness property is different to WS: it applies to composite optimization problems with objective: $F = f+h$ where $h$ must be smooth but $f$ may be nonsmooth. Unlike subgradient methods, in proximal gradient methods the nonsmooth part $f$ is addressed via its proximal operator. 

In recent times, convergence analyses for the subgradient method have focused on the objective function rather than the distance of the iterates from the optimal set. However in the early period of development, there were many works focusing on the distance (e.g. \cite{nedic2001convergence,shor2012minimization,poljak1978nonlinear,goffin1977convergence}). The subgradient method is not a descent method with respect to function values, however it is with respect to the distances to the optimal set. Thus the distance is a natural metric to study for the subgradient method.  Furthermore, for some applications, the distance to the solution set arguably matters more than the objective function value. For example in machine learning, the objective function is only a surrogate for the actual objective of interest -- expected prediction error. 

Without further assumptions, \cite[p. 167--168]{PolyakIntro} showed that the convergence rate of the distance of the iterates of the subgradient method to the optimal set can be made arbitrarily slow.
This is true even for smooth convex problems. In this case, gradient descent with a constant stepsize obtains an $O(1/k)$ \emph{objective function} convergence rate, however the iterates can be made to converge arbitrarily slowly to a minimizer. It is our use of HEB which allows us to derive less pessimistic convergence rates for the distance to the optimal set.

\section{The Key Recursion}\label{secKeyRec}
 In this section we derive the recursion which describes the evolution of the squared error $d(x_k,\calX_h)^2$ for the iterates of the standard subgradient method under HEB. The same recursion has been derived many times before for the special cases $\theta=\{1/2,1\}$ (e.g. \cite{goffin1977convergence,shor2012minimization,nedic2001convergence}).
\subsection{Assumptions}
\label{secMathDef}
 The optimality condition for Prob. (\ref{Prob1}) can be found in \cite[Prop. 26.5]{bauschke2011convex}. Note that we do not explicitly use this optimality criterion anywhere in our analysis.
For Prob. (\ref{Prob1}), throughout the manuscript we will assume that $\calC\subseteq\dom(\partial h)$, so that for any query point $x\in\calC$ it is possible to find a $g\in\partial h(x)$. If $h$ is convex and closed, the solution set $\calX_h=\{x:h(x)=h^*\}$ is convex and closed \cite{bauschke2011convex}.
Here are the precise assumptions we will use throughout the manuscript.

{\bf Assumption 3.} (Problem (\ref{Prob1})).
Assume $\calC$ is convex, closed, and nonempty. Assume $h$ is convex, closed, and satisfies HEB$(c,\theta)$. Assume $\calX_h$ is nonempty. Assume $\calC\subseteq\dom(\partial h)$. Assume there exists a constant $G$ such that $\|g\|\leq G$ for all $g\in\partial h(x)$ and $x\in\calC$. 

Throughout the manuscript let $\kappa\triangleq G/c$. 
\subsection{The Recursion under HEB}

 \begin{proposition}\label{Prop_keyRecur}
 Suppose Assumption 3 holds. Then for all $k\geq 1$ for the iterates $\{x_k\}$ of (\ref{iterSG})
 \begin{eqnarray}\label{KeyRecursion}
 d(x_{k+1},\calX_h)^2
 &\leq&
  d(x_k,\calX_h)^2
  -2\alpha_k c(d(x_k,\calX_h)^2)^{\frac{1}{2\theta}}
  +\alpha_k^2 G^2.
 \end{eqnarray}
 \end{proposition}
 \begin{proof}
 For the point $x_k$ let $x_k^*$ be the unique projection of $x_k$ onto $\calX_h$.  
 For $k\geq 1$,
 \begin{eqnarray}
 d(x_{k+1},\calX_h)^2
 &=&
 \|x_{k+1}-x_{k+1}^*\|^2
 \nonumber\\
 &\leq&
 \|x_{k+1}-x_k^*\|^2
 \nonumber\\\nonumber
 &\leq&
 d(x_k,\calX_h)^2
 -2\alpha_k\langle g_k,x_k - x_k^*\rangle
 +\alpha_k^2\|g_k\|^2
 \\\nonumber
 &\leq&
 d(x_k,\calX_h)^2
 -2\alpha_k\left(h(x_k)-h^*\right)
 +\alpha_k^2 G^2
 \\\nonumber
 &\leq&
 d(x_k,\calX_h)^2
 -2\alpha_k c(d(x_k,\calX_h)^2)^{\frac{1}{2\theta}}
 +\alpha_k^2 G^2.
 \end{eqnarray}
In the first inequality, we used the fact that $x_{k+1}^*$ is the closest point to $x_{k+1}$ in $\calX_h$. In the second inequality, we used the nonexpansiveness of the projection operator. In the third, we used the convexity of $h$ and in the final inequality we used the error bound. 
 \end{proof}
 
Let $e_k \triangleq d(x_k,\calX_h)^2$ and $\gamma=\frac{1}{2\theta}\in[\frac{1}{2},+\infty)$ then for all $k\geq 1$
 \begin{eqnarray}\label{ABiggy}
 0\leq e_{k+1}\leq e_k - 2\alpha_k c e_k^{\gamma} +\alpha_k^2 G^2.
 \end{eqnarray}
 The main effort of our analysis is in deriving convergence rates for this recursion for various stepsizes.
 
 We note that the key recursion (\ref{KeyRecursion}) can also be derived with different constants in the following extensions: 
 \begin{enumerate}
 	\item For $\theta=1$ a small (relative to $c$) amount of deterministic noise can be added to the subgradient  \cite{nedic2010effect}, 
 	\item 
 	A more general condition than HEB (with $\theta=1$), used in \cite{goffin1977convergence}, can be considered,
 	 	\item Instead of \eqref{iterSG} one can consider
 	the \emph{incremental} subgradient method \cite{nedic2001convergence}, 
 the \emph{proximal} subgradient method \cite{cruz2017proximal},
\begin{align*}
  x_{k+1}= \prox_{\alpha_k f}(x_k-\alpha_k g_k):\quad\forall k\geq 1, g_k\in\partial h(x_k),\,\, x_1\in \dom(\partial h),
\end{align*}
 for minimizing $F(x) = f(x)+h(x)$, so long as the composite function $F$ satisfies HEB and $\dom(\partial h)\subseteq\dom(f)$,
 or the \emph{relaxed} projected subgradient method:
  \begin{align*}
 x_{k+1}=(1-\theta_k)x_k + \theta_kP_\calC(x_k-\alpha_k g_k):\quad\forall k\geq 1, g_k\in\partial h(x_k),\,\, x_1\in \calC,
 \end{align*} 
 so long as $0<\underline{\theta}\leq\theta_k\leq 1$. 
  \end{enumerate}
 \label{secMy}
 
 Extensions 1-2 are discussed in more detail in Sec. \ref{secExtend}.
 
 \section{Constant Stepsize}\label{sec_const}
 Consider the projected subgradient method with \emph{constant}, or fixed, stepsize $\alpha$ given in Algorithm FixedSG.
  \begin{algorithm}
  \caption{(FixedSG)}\label{fix}
  \begin{algorithmic}[1]
  \REQUIRE $K>0$, $\alpha>0$, $x_1\in\calC$
  \FOR{$k=1,2,\ldots,K$}
           \STATE $x_{k+1}= P_{\calC}\left(x_k-\alpha g_k\right):\quad g_k\in\partial h(x_k)$
       \ENDFOR 
       \RETURN $x_{k+1}$
  \end{algorithmic}
  \end{algorithm}
 Previously it was known that if $\theta=1/2$ then this method achieves linear convergence to within a region of the solution set \cite{nedic2001convergence,karimi2016linear}. We show in the next theorem that linear convergence to within a certain region of $\calX_h$ occurs for any $\theta\in(0,1]$ provided $\alpha$ is sufficiently small.
 
 \begin{theorem}\label{ThmFix}
Suppose Assumption 3 holds. Let $e_* = \left(\frac{\alpha G^2}{2c}\right)^{2\theta}$. 
 \begin{enumerate}
 \item For all $k\geq 1$  the iterates of FixedSG satisfy
 \begin{eqnarray}\label{eqBounded}
 d(x_k,\calX_h)^2\leq \max\left\{d(x_1,\calX_h)^2,e_* + \alpha^2 G^2\right\}.
 \end{eqnarray}
 \item If $0<\theta\leq \frac{1}{2}$ 
 and
 \begin{eqnarray}\label{fixed0}
 0<\alpha 
 \leq 
 2^{\frac{1-2\theta}{2(1-\theta)}}\theta^{\frac{1}{2(1-\theta)}}
 G^{\frac{2\theta-1}{1-\theta}}
 c^{\frac{\theta}{\theta-1}}
 \end{eqnarray}
then
 for all $k\geq 1$ the iterates of FixedSG satisfy
 \begin{eqnarray}\label{fixed1}
d(x_k,\calX_h)^2 - e_*\leq q_1^{k-1}
(d(x_1,\calX_h)^2-e_*)
 \end{eqnarray}
 where
 \begin{eqnarray}\label{fixed2}
 q_1=\left(1-\frac{1}{\theta}\alpha c e_*^{\frac{1-2\theta}{2\theta}}\right)\in[0,1).
 \end{eqnarray}

 \item If  $\frac{1}{2}\leq\theta\leq 1$, suppose there exists $D\geq 0$ s.t.  $d(x_k,\calX_h)^2\leq D$ for all $k$, and the stepsize is chosen s.t.
 \begin{eqnarray}\label{linConv}
 0<\alpha\leq  \frac{\theta D^{1-\frac{1}{2\theta}}}{c},
 \end{eqnarray}
 then for all $k\geq 1$ the iterates of FixedSG satisfy
 \begin{eqnarray}\label{ghh}
 d(x_k,\calX_h)^2 - e_*\leq \max\{q_2^{k-1}(d(x_1,\calX_h)^2-e_*),\alpha^2 G^2\},
 \end{eqnarray}
 where
 $$
 q_2=
1-\frac{\alpha c D^{\frac{1}{2\theta}-1}}{\theta}\in[0,1).
 $$
 \end{enumerate}
 \end{theorem}
Note that in part 3 of Theorem \ref{ThmFix}, we assume the existence of a bound $D$ s.t. $d(x_k,\calX_h)^2\leq D$ for all $k\in\mathbb{N}$. Such a bound was provided in part 1 of the theorem. However for the sake of notational clarity we prove part 3 with a generic upper bound $D$. 

 \begin{proof}
 Recall our notation $e_k=d(x_k,\calX_h)^2$ and let $\gamma=\frac{1}{2\theta}$. 
 Returning to the main recursion (\ref{ABiggy}) derived in Prop. \ref{Prop_keyRecur} and replacing the stepsize with a constant yields
 \begin{eqnarray}\label{fixedRecurs}
 0\leq e_{k+1}\leq e_k - 2\alpha c e_k^\gamma + \alpha^2 G^2,
 \end{eqnarray}
 where $\gamma\geq \frac{1}{2}$. 
  The key to understanding the behavior of this recursion is to write it as
 \begin{eqnarray}
 e_{k+1}-e_*\leq e_k - e_* - 2\alpha c (e_k^\gamma - e_*^\gamma)\label{one}
 \end{eqnarray}
 where $e_* = \left(\frac{\alpha G^2}{2c}\right)^{\frac{1}{\gamma}}$.
  We will show that $e_k - e_*$ must go to $0$ and derive the convergence rate.

 \noindent
 {\bf \underline{Boundedness:}}

 \noindent 
 We first prove (\ref{eqBounded}), which says that $e_k$ is bounded. Considering (\ref{one}) we see that since $\alpha>0$ and $c>0$, if $e_k\geq e_*$ then $e_{k+1}\leq e_k$. On the other hand, if $e_k\leq e_*$, then (\ref{fixedRecurs}) yields $e_{k+1}\leq e_k+\alpha^2 G^2\leq e_*+\alpha^2 G^2$. Therefore
 \begin{eqnarray*}
 e_{k+1}\leq\max\{e_k,e_*+\alpha^2 G^2\}\leq\max\{e_1,e_*+\alpha^2 G^2\}.
 \end{eqnarray*}
 
 \noindent
 {\bf\underline{Case 1}: $\boldsymbol{\theta\leq\frac{1}{2}}$.}
 
 \noindent
 For $\theta\leq\frac{1}{2}$, $\gamma\geq 1$ and by the convexity of $t^\gamma$,
 \begin{eqnarray*}
 e_k^\gamma - e_*^\gamma &\geq& \gamma e_*^{\gamma-1}(e_k-e_*).
 \end{eqnarray*}
 Using this in (\ref{one}) along with the facts that $\alpha>0$ and $c>0$ yields 
 \begin{eqnarray*}
 e_{k+1}-e_*\leq  (1-2\alpha c\gamma e_*^{\gamma-1})(e_k - e_*).
 \end{eqnarray*}
 Thus so long as 
 \begin{eqnarray}
 1-2\alpha c\gamma e_*^{\gamma-1}\geq 0,\label{two}
 \end{eqnarray}
 we have $q_1\geq 0$ where $q_1$ is defined in \eqref{fixed2} and
 \begin{align*}
 e_{k+1}-e^*\leq q_1(e_k - e^*)
 \leq
 q_1^k(e_1 - e^*)
 \end{align*}
where the second inequality comes from recursing. This proves \eqref{fixed1}.

Simplifying (\ref{two}) yields
 \begin{eqnarray*}
 2\alpha c \gamma e_*^{\gamma-1}
 &\leq& 1
 \\
 &\implies&
\alpha c \gamma  \left(\frac{\alpha G^2}{2c}\right)^{\frac{\gamma-1}{\gamma}}
 \leq 2^{-1}
 \\
 &\implies&
 \alpha
 \leq 
 \left(
 \frac{1}{\gamma}G^{\frac{2(1-\gamma)}{\gamma}}2^{-\frac{1}{\gamma}}
 c^{-\frac{1}{\gamma}}
 \right)^{\frac{\gamma}{2\gamma-1}}
 \end{eqnarray*}
 which  is equivalent to  (\ref{fixed0}).
 
 \noindent{\bf\underline{Case 2}: $\boldsymbol{\theta\geq\frac{1}{2}}$.}
 
 \noindent
 For $\theta\in[\frac{1}{2},1]$, $\gamma\in[\frac{1}{2},1]$, which implies by concavity
 \begin{eqnarray*}
 e_*^\gamma- e_k^\gamma
 \leq 
 \gamma e_k^{\gamma-1}(e_*-e_k).
 \end{eqnarray*}
 Therefore
 \begin{eqnarray*}
 e_k^\gamma-e_*^{\gamma}
 \geq
 \gamma  e_k^{\gamma-1}
 (e_k-e_*).
 \end{eqnarray*}
 Substituting this inequality into (\ref{one}) and again using $\alpha>0$ and $c>0$ yields
 \begin{eqnarray*}
 e_{k+1}-e_*&\leq& 
 e_k - e_* - 2\alpha c\gamma e_k^{\gamma-1} (e_k - e_*).
 \end{eqnarray*}
  Now if $e_*\leq e_k$, then since $e_k\leq D$,
 \begin{eqnarray*}\label{pas}
 e_{k+1}-e_*&\leq& 
 (1 - 2\alpha c\gamma D^{\gamma-1}) (e_k - e_*) = q_2(e_k - e_*).
 \end{eqnarray*}
 So long as 
 \begin{eqnarray*}
 1>1-2\alpha c\gamma  D^{\gamma-1}\geq 0
 \end{eqnarray*}
(which is implied by (\ref{linConv})), we have $q_2\in [0,1)$. On the other hand if $e_k\leq e_*$ then, using (\ref{fixedRecurs}), $e_{k+1}\leq e_*+ \alpha^2 G^2$. Thus for all $k\geq 1$
\begin{eqnarray*}
e_{k+1}-e_*\leq 
\max\left\{
q_2(e_k-e_*),\alpha^2 G^2
\right\}
.
\end{eqnarray*}
 Iterating this recursion and using the fact that $q_2\in[0,1)$ yields (\ref{ghh}).
 
\end{proof}

 \section{Iteration Complexity for Constant Stepsize}\label{sec_itercomp}
 Using the results of the previous section we can derive the iteration complexity of a constant stepsize for finding a point such that $d(x_k,\calX_h)^2\leq\epsilon$. 
   The basic idea in the following theorem is to pick $\alpha=O(\epsilon^{\frac{1}{2\theta}})$, so that $e_*$ defined in Theorem \ref{ThmFix} is equal to $\epsilon$. Then the iteration complexity can be determined from the linear convergence rate of $d(x_k,\calX_h)^2$ to $e_*$.

  \begin{theorem}\label{ThmFixIterComp}
Suppose Assumption 3 holds. Choose $\epsilon>0$ and set 
\begin{eqnarray}
\label{stepsizeFixed}
   \alpha&=&\frac{2c\epsilon^{\frac{1}{2\theta}}}{G^2}.
\end{eqnarray}
  \begin{enumerate}
  \item If $0<\theta\leq\frac{1}{2}$, and
   \begin{eqnarray}\label{epsbound}
   0<\epsilon&\leq&\left(\frac{\theta \kappa^2}{2}\right)^{\frac{\theta}{1-\theta}},
   \end{eqnarray}
     then for the iterates of FixedSG,
      $$d(x_{k+1},\calX_h)^2\leq 2\epsilon$$ for all $k\geq K$ where 
   \begin{eqnarray}
  K&\triangleq& 
  \frac{1}{2}\theta \kappa^2 \ln\left(\frac{d(x_1,\calX_h)^2}{\epsilon}\right)\epsilon^{1-\frac{1}{\theta}}
.\nonumber
 \end{eqnarray}
   \item For $\frac{1}{2}\leq \theta\leq 1$, assume $\hat{D}>0$ and $\epsilon>0$ are chosen s.t. 
   \begin{align}\label{DinitBound}
  d(x_1,\calX_h)^2 &\leq \hat{D}
  \\
\epsilon&\leq\min
\left\{\frac{\hat{D}}{2},\left(\frac{\theta \kappa^2}{2}\right)^{2\theta}\hat{D}^{2\theta-1}
\right\}.
\label{Dbound}
    \end{align} 
If $\theta<1$ we further require
\begin{eqnarray} 
         \label{epsBound2}
  \epsilon &\leq&\left(\frac{\kappa^2}{4}\right)^{\frac{\theta}{1-\theta}}
   \end{eqnarray}
  and if $\theta=1$, we require $\kappa\geq 2$. Then
for the iterates of FixedSG, 
$$d(x_{k+1},\calX_h)^2\leq 2\epsilon$$ for all $k\geq K$, where
   \begin{eqnarray}
    \label{requiredK}
       K &\triangleq& 
       \frac{1}{2}\theta\kappa^2 \hat{D}^{1-\frac{1}{2\theta}} \ln\left(\frac{d(x_1,\calX_h)^2}{\epsilon}\right)\epsilon^{-\frac{1}{2\theta}}.
   \end{eqnarray}
  \end{enumerate}
  \end{theorem}
  
  \begin{proof} We consider the two cases, $\theta\leq1/2$ and $\theta\geq 1/2$, separately.
  
  \noindent{\bf\underline{Case 1}: $\boldsymbol{\theta\leq\frac{1}{2}}$.}
  
  \noindent
  From Theorem \ref{ThmFix}, the convergence factor in the constant stepsize case is $q_1=1-\frac{\alpha c}{\theta} e_*^{\frac{1}{2\theta}-1}$ where $e_*=\left(\frac{\alpha G^2}{2c}\right)^{2\theta}=\epsilon$ for this choice of $\alpha$ given in \eqref{stepsizeFixed}. Recall the notation $e_k=d(x_k,\calX_h)^2$. 
  Since $\epsilon$ satisfies (\ref{epsbound}), $0\leq q_1<1$. 
  Thus from Theorem \ref{ThmFix} we know that for all $k\geq 1$
  \begin{align*}
  e_{k+1} - e_*\leq q_1^k(e_1 - e^*)\leq q_1^k e_1
  \end{align*}
  which implies
  \begin{align}\label{eqNewalign}
  \max\{e_{k+1}-e_*,0\}\leq q_1^k e_1.
  \end{align}
  This means that  
  $$
  \ln(\max\{0,e_{k+1}-e_*\})\leq k\ln q_1 + \ln e_1
  $$
  using the convention, $\ln(0)=-\infty$. 
  Thus $e_{k+1}-e_*\leq\epsilon$ is implied by
  $$
  k\ln q_1 + \ln e_1\leq \ln\epsilon
  \iff k\geq \frac{\ln\frac{e_1}{\epsilon}}{\ln\frac{1}{q_1}}.
  $$
    Now 
  \begin{eqnarray*}
  \ln q_1=\ln\left(1-\frac{\alpha c}{\theta}e_*^{\frac{1}{2\theta}-1}\right)\leq -\frac{\alpha c}{\theta} e_*^{\frac{1}{2\theta}-1}\iff \ln\frac{1}{q_1}\geq \frac{\alpha c}{\theta} e_*^{\frac{1}{2\theta}-1}.
  \end{eqnarray*}
  Therefore if
  \begin{eqnarray*}
  k\geq 
  \frac{\theta\ln\frac{e_1}{\epsilon}}{\alpha c e_*^{\frac{1}{2\theta}-1}}
  =
  \frac{\theta G^2\ln\frac{e_1}{\epsilon}}{2c^2\epsilon^{\frac{1}{\theta}-1}}
  =
  \frac{1}{2}\theta\kappa^2\ln\left(\frac{e_1}{\epsilon}\right)\epsilon^{1-\frac{1}{\theta}}
  \end{eqnarray*}
  then using the fact that for this choice of $\alpha$, $e_*=\epsilon$, we arrive at
  \begin{eqnarray*}
  e_{k+1}\leq e_* +\epsilon = 2\epsilon.
  \end{eqnarray*}

  \noindent{\bf\underline{Case 2}: $\boldsymbol{\theta\geq \frac{1}{2}}$.}
  
  \noindent
  As before, $\alpha = \frac{2c\epsilon^{\frac{1}{2\theta}}}{G^2}$ which implies $e_*=\epsilon$.
  First note that by Part 1 of Theorem \ref{ThmFix}, 
  \begin{eqnarray*}
  d(x_k,\calX_h)^2
  &\leq& 
  \max\{d(x_1,\calX_h)^2,e_*+\alpha^2 G^2\}
  \\
  &=&
  \max\left\{d(x_1,\calX_h)^2,\epsilon+\frac{4 c^2}{G^2}\epsilon^{\frac{1}{\theta}}\right\}
  \\
  &\leq&
  \max\{d(x_1,\calX_h)^2,2\epsilon\}
  \\
  &\leq& \hat{D}
  \end{eqnarray*}
   for all $k\geq 1$, where in the second inequality we used \eqref{DinitBound} and (\ref{epsBound2}) for the case $\theta<1$, and $\kappa\geq 2$ for when $\theta=1$. Therefore $\hat{D}$ is a valid upper bound for the sequence $\{d(x_k,\calX_h)^2\}$ and can be used in place of $D$ in Theorem \ref{ThmFix}. 
    Now from Theorem \ref{ThmFix} the convergence factor is
   \begin{eqnarray*}
   	q_2=1-\frac{\alpha c}{\theta}\hat{D}^{\frac{1}{2\theta}-1}
   \end{eqnarray*}
   which is greater than or equal to $0$ (and less than $1$) because $\epsilon$ satisfies (\ref{Dbound}).
  Recalling (\ref{ghh}) we see that 
  \begin{eqnarray}\label{newcasesss}
  e_{k+1}\leq \max\{e_*+q_2^{k}(d(x_1,\calX_h)^2-e_*),e_*+\alpha^2 G^2\}.
  \end{eqnarray}
  We have already shown that the second argument in the $\max$ above is upper bounded by $2\epsilon$. 
Consider the first argument in the $\max$ in (\ref{newcasesss}). 
 Now $q_2^{k}(d(x_1,\calX_h)^2-e_*)\leq q_2^{k} e_1$, thus
this argument can be dealt with the same way as Case 1 for $\theta\leq1/2$, except for a different convergence factor. 
 Thus we observe
 \begin{eqnarray*}
  \ln q_2=\ln\left(1-\frac{\alpha c}{\theta} \hat{D}^{\frac{1}{2\theta}-1}\right)\leq -\frac{\alpha c}{\theta}  \hat{D}^{\frac{1}{2\theta}-1}\iff \ln\frac{1}{q_2}\geq \frac{\alpha c}{\theta} \hat{D}^{\frac{1}{2\theta}-1}.
  \end{eqnarray*}
  Therefore if 
  \begin{eqnarray*}
  k
  \geq
   \frac{\theta G^2 \hat{D}^{1-\frac{1}{2\theta}}}{2c^2}\ln\left(\frac{e_1}{\epsilon}\right)\epsilon^{-\frac{1}{2\theta}}
  \end{eqnarray*}
 then the first argument in the $\max$ in (\ref{newcasesss}) is upper bounded by $2\epsilon$.

  \end{proof} 
  
  Rather surprisingly, Theorem \ref{ThmFixIterComp} shows that a restarting strategy is not necessary for $\theta\leq\frac{1}{2}$. This is because for $\theta\leq\frac{1}{2}$ the iteration complexity for a constant stepsize is equal to the complexity of RSG derived in \cite{yang2015rsg}. It is also matched by the optimal decaying stepsize we will derive in Sec. \ref{sec_sum}. 
     To compare with RSG in more detail, \cite{yang2015rsg} showed that RSG requires $O(\epsilon'^{2(\theta-1)})$ iterations (suppressing constants and a $\ln\frac{1}{\epsilon}$ factor) to achieve $h(x)-h^*\leq\epsilon'$. Now, using the error bound, in order to guarantee $d(x_k,\calX_h)^2\leq\epsilon$, we need $h(x)-h^*\leq \epsilon'=\epsilon^{\frac{1}{2\theta}}$. Using this in the iteration complexity from \cite{yang2015rsg} yields the expression $O(\epsilon^{1-\frac{1}{\theta}})$, which is the same as what we derived for the constant stepsize for $\theta\leq 1/2$. However, for $\theta>\frac{1}{2}$, RSG, our DS-SG method, and our optimal decaying stepsize are significantly faster than the constant stepsize choice.
     For $\theta=1/2$, the iteration complexity of the constant stepsize derived in Theorem \ref{ThmFixIterComp} depends on $\ln d(x_1,\calX_h)$, and has the same dependence on $\epsilon$ as the other methods. This remarkable property makes it preferable to the other more sophisticated methods in this case. 
    
    The comparison with the classical result for the subgradient method is as follows. It is easy to show that for the subgradient method with a constant stepsize $\alpha$:
    \begin{eqnarray*}
    \frac{1}{k}\sum_{i=1}^k(h(x_i)-h^*)
    \leq
    \frac{d(x_1,\calX_h)^2}{2\alpha k}+\frac{\alpha}{2}G^2.
    \end{eqnarray*}
    Setting 
    $$\alpha=\frac{c\epsilon^{\frac{1}{2\theta}}}{G^2}
    $$
     and 
    \begin{eqnarray*}\label{classicalIter}
   k\geq \kappa^2 d(x_1,\calX_h)^2\epsilon^{-1/\theta}
    \end{eqnarray*} 
    implies 
    $$
    h(x_k^{av})-h^*\leq \frac{1}{k}\sum_{i=1}^k(h(x_i)-h^*)\leq  c\epsilon^{1/2\theta}
    $$ 
    where $x_k^{av}=\frac{1}{k}\sum_{i=1}^k x_i$. Now using the error bound, this yields $d(x_k^{av},\calX_h)^2\leq\epsilon$. With respect to $\epsilon$, this classical iteration complexity is clearly worse than the result of Theorem \ref{ThmFix} for all $\theta\in(0,1]$. Furthermore, the dependence on $d(x_1,\calX_h)$ is worse. For $\theta\leq1/2$, the fixed stepsize depends on $\ln d(x_1,\calX_h)$, whereas the classical stepsize has iteration complexity which depends linearly on $d(x_1,\calX_h)$. 
    
    
    
    We note that as $\theta\to 0$ the iteration complexity can be made arbitrarily large. This is not suprising, as it has been proved in \cite[p. 167-168]{PolyakIntro} that the convergence rate of $x_k\to x^*$ can be made arbitrarily bad for gradient methods.

 \section{A ``Descending Stairs" Stepsize with Better Iteration Complexity for $1/2\leq \theta\leq 1$}
 \label{secRest}
 \subsection{The Method}
  In this section we propose a new stepsize for the subgradient method (DS-SG) which obtains a better iteration complexity than the fixed stepsize for functions satisfying HEB with $1/2\leq\theta\leq 1$. In fact for $\theta=1$ the iteration complexity is logarithmic, i.e. $O(\ln\frac{1}{\epsilon})$. 
  The basic idea is to use a constant stepsize in the subgradient method and every $K$ iterations reduce the stepsize by a factor of $\beta_{ds}^{\frac{1}{2\theta}}>1$. Also the number of iterations $K$ increases by a factor $\beta_{ds}^{\frac{1}{\theta}-1}$. Our analysis allows us to determine good choices for the initial stepsize and number of iterations which lead to an improved rate. 

 The algorithm is similar to RSG \cite{yang2015rsg}. However our  method has some important advantages, which will be discussed in Sec. \ref{secDD_discuss}, and a different analysis. As was mentioned earlier, the method of \cite[Sec. V]{supittayapornpong2016staggered} is essentially a special case of DS-SG for $\theta=1$.

  DS-SG requires an upper bound on the distance of the starting point to the solution, i.e. $\Omega_1\geq d(x_{\text{init}},\calX_h)^2$. If $\calC$ is bounded then one can use the diameter of $\calC$. If a lower bound on the optimal value is known, i.e. $h_{l}\leq h^*$, then by the error bound $d(x_1,\calX_h)\leq c^{-\theta}\left(h(x_1)-h^*\right)^\theta\leq c^{-\theta}\left(h(x_1)-h_l\right)^\theta$ implies we can use $\Omega_1=c^{-2\theta}\left(h(x_1)-h_l\right)^{2\theta}$.
 
 \begin{algorithm}
 
   \caption{(DS-SG) Descending Stairs Subgradient Method for $1/2\leq \theta\leq 1$}
   \begin{algorithmic}[1]
   \REQUIRE $\beta_{ds}$, $M$, $x_{\text{init}}$, $\Omega_1$, $G$, $c$, $\theta$.
   \STATE  $\kappa=\frac{G}{c}$ 
   \STATE \label{Kline1} $\tilde{K}_1=
     \theta
 \kappa^2\beta_{ds}^{\frac{1}{2\theta}}\ln\left(2\beta_{ds}\right)
      \Omega_1^{1-\frac{1}{\theta}}
         $
   \STATE $K_1 = \lceil\tilde{K}_1\rceil$
   \STATE\label{alphaline1} $\alpha(1)=\frac{2c}{G^2}\left(\frac{\Omega_1}{2\beta_{ds}}\right)^{\frac{1}{2\theta}}$
   \STATE $\hat{x}_0=x_{\text{init}}$
   \FOR{$m=1,2,\ldots,M$}
      \STATE    $\hat{x}_m = \text{FixedSG}(K_m,\alpha(m),\hat{x}_{m-1})$\label{LineXhat}
      \STATE \label{alphaline2} $\alpha(m+1) = \beta_{ds}^{-\frac{1}{2\theta}}\alpha(m)$
      \STATE $K_{m+1}=\left\lceil\beta_{ds}^{\frac{m(1-\theta)}{\theta}}\tilde{K}_1\right\rceil$\label{Kline2}
   \ENDFOR
   \RETURN $\hat{x}_M$ 
   \end{algorithmic}
   \label{ReSG}
\end{algorithm}
 
 \begin{theorem}\label{thmRestart}
Suppose Assumption 3 holds and $\frac{1}{2}\leq \theta\leq 1$. Choose $x_{\text{init}}\in\calC$ and $\Omega_1$ such that $d(x_{\text{init}},\calX_h)^2\leq\Omega_1$. If $\theta<1$, choose $\beta_{ds}>1$ so that
\begin{eqnarray}\label{beta_bound}
\beta_{ds}&\geq& 
\max\left\{\frac{1}{2}
\left(\frac{\kappa^2}{4}\right)^{\frac{\theta}{\theta-1}}\Omega_1,
\theta^{-2\theta}\kappa^{-4\theta}\Omega_1^{2(1-\theta)}
\right\}.
\end{eqnarray}
If $\theta=1$, assume $\kappa\geq 2$ and choose any $\beta_{ds}>1$.
  Fix $\epsilon>0$ and choose $M\geq\left\lceil\frac{\ln\frac{\Omega_{1}}{\epsilon}}{\ln\beta_{ds}}\right\rceil$. Then for $\hat{x}_M$ returned by Algorithm DS-SG, $d(\hat{x}_M,\calX_h)^2\leq \epsilon$. The iteration complexity is as follows:
  \begin{enumerate}
  \item If $\theta=1$ this requires fewer than
  \begin{eqnarray}
     &&\left(\beta_{ds}^{\frac{1}{2}}  \kappa^2 \ln(2\beta_{ds})+1\right)
   \left(\frac{\ln\frac{\Omega_1}{\epsilon}}{\ln\beta_{ds}}+1\right)
  \label{Th1ResultA}\label{Th1Result}
  \end{eqnarray}
  subgradient evaluations. This simplifies to 
  \begin{align}\label{bigO1}
  O\left(\kappa^2\ln\frac{\Omega_1}{\epsilon}\right)
  \end{align}
  as $\kappa,\Omega_1\to\infty$, and $\epsilon\to 0$.
  \item If $\frac{1}{2}\leq\theta<1$, this requires fewer than
  \begin{eqnarray}
\frac{\theta
	\beta_{ds}^{\frac{3}{2\theta}-1}\ln(2\beta_{ds})}{\beta_{ds}^{\frac{1}{\theta}-1}-1
} 
\kappa^2
\epsilon^{1-\frac{1}{\theta}}
+
\frac{\ln\frac{\Omega_1}{\epsilon}}{\ln\beta_{ds}}+1  \label{scoreA}
  \end{eqnarray}
    subgradient evaluations. If $\kappa$ is chosen large enough so that $\Omega_1 = O(\kappa^{\frac{2\theta}{1-\theta}})$,
this simplifies to 
  \begin{align}\label{bigO}
  O\left(
\max\{\kappa^2,\Omega^{\frac{1}{\theta}-1}\}\epsilon^{1-\frac{1}{\theta}}
  \right)
  \end{align}
  as $\kappa,\Omega_1\to\infty$, and $\epsilon\to 0$.

      \end{enumerate}
 \end{theorem}
 \begin{proof}
 We need some new notation. For $\hat{x}_m$ defined in line \ref{LineXhat} of DS-SG, let $\hat{e}_m=d(\hat{x}_m,\calX_h)^2$. We will use a sequence of tolerances $\{\epsilon_m\}$ defined as $\epsilon_m=\beta_{ds}^{-m}\Omega_1$. Another sequence $\{D_m\}$ is chosen as 
 $$D_m = 2\beta_{ds}\epsilon_m.$$ 
   For each $m\geq 1$, the set $\{\epsilon_m/2,D_m,\alpha(m)\}$ will be used in statement 2 of Theorem \ref{ThmFixIterComp} in place of $\{\epsilon,\hat{D},\alpha\}$. Furthermore we will show that $K_m$ is greater than the corresponding expression for $K$ in \eqref{requiredK}. This will show that $\hat{e}_m\leq 2(\epsilon_m/2)=\epsilon_m$.

   We now show that $\{\epsilon_m/2,D_m,\alpha(m)\}$ satisfies (\ref{stepsizeFixed}), \eqref{DinitBound}, (\ref{Dbound}),  and (\ref{epsBound2}), and that $K_m$ is greater than $K$ given in (\ref{requiredK}). Now
the stepsize $\alpha(m)$, defined on lines \ref{alphaline1} and \ref{alphaline2} of DS-SG, can be written as
   \begin{eqnarray*}
   \alpha(m) = \frac{2 c}{ G^2}\left(\frac{\epsilon_m}{2}\right)^{\frac{1}{2\theta}}.
   \end{eqnarray*}
       Thus $\alpha(m)$ satisfies (\ref{stepsizeFixed}) for all $m\geq 1$. 
          Next we prove that for $\frac{1}{2}\leq \theta<1$, condition (\ref{beta_bound}) ensures that \eqref{Dbound}--(\ref{epsBound2}) are satisfied for all $m\geq 1$. We also show that for $\theta=1$, \eqref{Dbound} is implied by $\kappa\geq 2$ (recall that \eqref{epsBound2} is only required for $\frac{1}{2}\leq\theta<1$). 
          
             To establish \eqref{Dbound}, we will prove that both arguments in the $\min$ in \eqref{Dbound} individually satisfy the inequality when $\hat{D}$ and $\epsilon$ are replaced by $D_m$ and $\epsilon_{m}/2$.
          Since $\beta_{ds}>1$, it is clear that the first argument in the $\min$ in \eqref{Dbound} satisfies the inequality.
Now for the second argument in the $\min$ in (\ref{Dbound}) to satisfy the inequality we require
\begin{eqnarray*}
	\frac{\epsilon_m}{2}\leq 
	\left(\frac{\theta\kappa^2}{2}\right)^{2\theta}
	D_m^{2\theta-1} = 
	\frac{1}{2}\left(\theta\kappa^2\right)^{2\theta}
	\beta_{ds}^{2\theta-1}\epsilon_m^{2\theta-1}.
\end{eqnarray*}
Using $\epsilon_m=\beta_{ds}^{-m}\Omega_1$ and rearranging this yields
\begin{eqnarray}
	\beta_{ds}^{2m(1-\theta)+2\theta-1}\geq 
	\theta^{-2\theta}\kappa^{-4\theta}
	\Omega_1^{2(1-\theta)}.\label{missed}
\end{eqnarray}
In order to hold for all $m\geq 1$ it suffices to show it holds for $m=1$, which is implied by the second argument in the max in (\ref{beta_bound}).
In the case $\theta=1$,
\eqref{missed} reduces to 
\begin{align*}
\beta_{ds}\geq \frac{1}{\kappa^4}.
\end{align*}
Since $\kappa\geq 2$, any $\beta_{ds}>1$ satisfies this.

           Now \eqref{epsBound2} is only required when $\frac{1}{2}\leq\theta<1$. In this case, \eqref{epsBound2} requires that
   \begin{eqnarray*}
   \frac{\epsilon_m}{2} = \frac{1}{2}\beta_{ds}^{-m}\Omega_1
      \leq
      \left(\frac{\kappa^2}{4}\right)^{\frac{\theta}{1-\theta}}.
   \end{eqnarray*}
   In order for this to be satisfied for all $m$, it suffices to show that it holds for $m=1$. This is implied by the first argument in the max function in (\ref{beta_bound}). 

    Finally we prove by induction that \eqref{DinitBound} holds and that $K_m$ is greater than $K$ defined in (\ref{requiredK}). For $m=1$, $D_1$ clearly satisfies \eqref{DinitBound}.   Also $K_1$, given in Line 1 of Algorithm DS-SG, satisfies (\ref{requiredK}).  Altogether this implies $\hat{e}_1\leq\epsilon_1$ by Theorem \ref{ThmFixIterComp}.

   Next, assume \eqref{DinitBound}  is true and $K_{m-1}$ is greater than $K$ in \eqref{requiredK} at iteration $m-1$. Since we have established \eqref{stepsizeFixed}, \eqref{Dbound}, and \eqref{epsBound2} hold for all $m\geq 1$, part 2 of Theorem \ref{ThmFix} implies that $\hat{e}_{m-1}\leq\epsilon_{m-1}$. At iteration $m$, FixedSG is initialized at $\hat{x}_{m-1}$, and $d(\hat{x}_{m-1},\calX)^2\leq\epsilon_{m-1}$, thus 
   $$
   D_m=2\beta_{ds}\epsilon_m = 2\epsilon_{m-1}\geq d(\hat{x}_{m-1},\calX)^2
   $$ 
   which establishes \eqref{DinitBound} at iteration $m$.
   Next, substituting $D_m$ and $\epsilon_m/2$ in for $\hat{D}$ and $\epsilon$ in \eqref{requiredK}, we see that $K_m$ needs to be greater than
    \begin{eqnarray*}
\frac{1}{2}\theta\kappa^2
    \ln\left(\frac{2d(\hat{x}_{m-1},\calX_h)^2}{\epsilon_m}\right)
    D_m^{1-\frac{1}{2\theta}}(\epsilon_m/2)^{-\frac{1}{2\theta}}\label{thisOne}
    \end{eqnarray*}
    which is indeed true since
$K_m$ can be re-expressed as
    \begin{align*}
    K_m&=\left\lceil 
    \theta\kappa^2\beta_{ds}^{\frac{1}{2\theta}}\ln(2\beta_{ds})\Omega_1^{1-\frac{1}{\theta}}\beta_{ds}^{-(m-1)(1-\frac{1}{\theta})}
    \right\rceil 
    \\    
    &=
    \left\lceil\theta\kappa^2
    \beta_{ds}^{1-\frac{1}{2\theta}}\ln\left(2\beta_{ds}\right)
    \epsilon_m^{1-\frac{1}{\theta}}\right\rceil
    \\
    &\geq 
    \frac{1}{2}\theta\kappa^2
    \ln\left(\frac{2d(\hat{x}_{m-1},\calX_h)^2}{\epsilon_m}\right)
    (2\beta_{ds}\epsilon_m)^{1-\frac{1}{2\theta}}(\epsilon_m/2)^{-\frac{1}{2\theta}}.
    \end{align*}

    We have shown that $\{\epsilon_m/2,D_m,\alpha(m)\}$ satisfies (\ref{stepsizeFixed}), \eqref{DinitBound}, (\ref{Dbound}), and (\ref{epsBound2}), and that $K_m$ is greater than $K$ defined in (\ref{requiredK}).
 Thus by part 2 of Theorem \ref{ThmFix}, for all $m\geq 1$ $\hat{e}_m\leq 2(\epsilon_m/2)=\epsilon_{m}$. 
  Finally the choice $M = \left\lceil \frac{\ln\frac{\Omega_1}{\epsilon}}{\ln\beta_{ds}}\right\rceil$ implies $\epsilon_M=\beta_{ds}^{-M}\Omega_1\leq\epsilon$.
  
 If $\theta=1$, the total number of subgradient evaluations is
 \begin{eqnarray*}
 M K_1
&\leq& 
 \left(\kappa^2 \beta_{ds}^{\frac{1}{2}}\ln(2\beta_{ds})+1\right)
\left(\frac{\ln\frac{\Omega_1}{\epsilon}}{\ln\beta_{ds}}+1\right)
 \end{eqnarray*}
where we have used $\lceil x \rceil< x+1$. 
Further note that for $\theta=1$, $\beta_{ds}$ is a constant that can be chosen independently of $\kappa, \Omega$, and $\epsilon$,  which implies \eqref{bigO1}.

 We now establish the iteration complexity when $\frac{1}{2}\leq\theta<1$.
For $m\geq 0$, let 
\begin{align}\label{defKtild}
\tilde{K}_{m+1} =      \beta_{ds}^{\frac{m(1-\theta)}{\theta}}\theta
\kappa^2\beta_{ds}^{\frac{1}{2\theta}}\ln\left(2\beta_{ds}\right)
\Omega_1^{1-\frac{1}{\theta}}
\end{align}
 
then $K_m = \lceil \tilde{K}_m\rceil$ where $K_m$ is defined on Line \ref{Kline2} of Algorithm \ref{ReSG}. 
If $\theta<1$ 
the total number of subgradient evaluations is
 \begin{eqnarray}\nonumber
    K_1 + K_2 +\ldots+ K_M
    &=&
 \lceil \tilde{K}_1\rceil +\lceil \tilde{K}_2\rceil +\ldots +\lceil \tilde{K}_M\rceil
 \\\nonumber
 &<&
   \tilde{K}_1 + \tilde{K}_2 +\ldots+ \tilde{K}_M
  +
  M
  \\\nonumber 
 &=&
 \tilde{K}_1
\left(
1+\beta_{ds}^{\frac{1}{\theta}-1}
+
(\beta_{ds}^{\frac{1}{\theta}-1})^2
+
\ldots
+
(\beta_{ds}^{\frac{1}{\theta-1}})^{M-1}
\right)
\\\nonumber
&&+M
\\\nonumber
&=&
\tilde{K}_1\frac{(\beta_{ds}^{\frac{1}{\theta}-1})^M - 1}{\beta_{ds}^{\frac{1}{\theta}-1}-1}
+
M
\\\label{totalComp} 
&\leq&
\tilde{K}_1\frac{(\beta_{ds}^{\frac{1}{\theta}-1})^M}{\beta_{ds}^{\frac{1}{\theta}-1}-1}
+M
.
 \end{eqnarray}
Now since
\begin{align}
M\leq \frac{\ln\frac{\Omega_1}{\epsilon}}{\ln\beta_{ds}}+1\label{Mupper}
\end{align} 
it follows that
\begin{eqnarray}
(\beta_{ds}^{\frac{1}{\theta}-1})^M\leq 
\beta_{ds}^{\frac{1}{\theta}-1}
\left(\frac{\Omega_1}{\epsilon}\right)^{\frac{1}{\theta}-1}.\label{do}
\end{eqnarray}
Finally, substitute \eqref{Mupper}, (\ref{do}), and the expression for $\tilde{K}_1$ into  (\ref{totalComp}) to obtain the iteration complexity 
\begin{eqnarray}\label{eq25new} 
\frac{\theta
	\beta_{ds}^{\frac{3}{2\theta}-1}\ln(2\beta_{ds})}{\beta_{ds}^{\frac{1}{\theta}-1}-1
} 
\kappa^2
\epsilon^{1-\frac{1}{\theta}}
+
\frac{\ln\frac{\Omega_1}{\epsilon}}{\ln\beta_{ds}}+1
\end{eqnarray}
total subgradient evaluations,
which is \eqref{scoreA}. 

Now onto \eqref{bigO}. We derive the limiting behavior of \eqref{eq25new} as $\epsilon\to 0$, and $\Omega_1$ and $\kappa$ approach $\infty$. 
In order to do this, we will prove that if $\Omega_1 = O(\kappa^{\frac{2\theta}{1-\theta}})$, the requisite lower bound on $\beta_{ds}$ in \eqref{beta_bound} is $O(1)$, which implies that $\beta_{ds}$ can be chosen as an $O(1)$ constant. 
 If $\kappa$ is too small, then it is enlarged to size $\Theta(\Omega^{\frac{1-\theta}{2\theta}})$ so that this does hold.
 
Considering each argument in the $\max$ in \eqref{beta_bound}, the first is
\begin{align*}
\frac{1}{2}
\left(\frac{\kappa^2}{4}\right)^{\frac{\theta}{\theta-1}}\Omega_1
=
O\left( 
\kappa^{\frac{2\theta}{\theta-1}}\Omega_1
\right) 
 =  O(1)
\end{align*}
and the second is
\begin{align*}
\theta^{-2\theta}\kappa^{-4\theta}\Omega_1^{2(1-\theta)}
=
O
\left(
\kappa^{-4\theta}\Omega_1^{2(1-\theta)}
\right)
 = O(1)
\end{align*}
where we have used the assumption that $\Omega_1 = O(\kappa^{\frac{2\theta}{1-\theta}})$.
Since $\beta_{ds}$ is $O(1)$ under this assumption, \eqref{eq25new} implies the number of subgradient evaluations behaves as 
$
O(\kappa^2\epsilon^{1-\frac{1}{\theta}}).
$
Since $\kappa$ may have to be enlargened to $\Theta(\Omega^{\frac{1-\theta}{2\theta}})$, this implies the subgradient evaluations actually behave as
$
O(\max\{\kappa^2,\Omega^{\frac{1-\theta}{\theta}}\}\epsilon^{1-\frac{1}{\theta}}), 
$
which yields \eqref{bigO}.

   \end{proof}

  \subsection{Discussion}\label{Discuss}\label{secDD_discuss}
%
  
The optimal choice for $\beta_{ds}$ can be found by minimizing the iteration complexities given in (\ref{Th1ResultA}) and (\ref{scoreA})  w.r.t. $\beta_{ds}$. However the closed form expression is complicated and not particularly enlightening. Solving it numerically, we find it is typically between $2$ and $2.5$.

  
Regarding RSG \cite{yang2015rsg}, the iteration complexity is very similar to ours, even though the analysis is different. There are several points to note in comparing the two. First is that their error metric is $h(x)-h^*$. 
 On the other hand our error metric is $d(x_k,\calX_h)^2$. Furthermore their iteration complexity is for finding $h(x)-h^*\leq 2\epsilon$. To do a fair comparison, we can convert their error metric to $d(x_k,\calX_h)^2$ by using $\epsilon'=2^{-1}\epsilon^{\frac{1}{2\theta}}$ in their iteration complexity. As we mentioned earlier, their iteration complexity is $O(\epsilon'^{2(\theta-1)}\ln\frac{1}{\epsilon'})$. Thus, if we make the substitution, we see that their iteration complexity is the same as ours except they have an extra $\log\frac{1}{\epsilon}$ term.
 The dependence on $\kappa = G/c$ is the same. 
 
 With respect to their algorithm implementation as given in \cite[Algorithm 2]{yang2015rsg}, the major difference to DS-SG is that \cite{yang2015rsg} requires averaging to be done after every inner loop. As mention before, this may be undesirable on problems where nonergodic methods are preferable. For instance, in problems where $\calC$ enforces sparsity or low-rank, the averaging phase spoils this property \cite{davis2017three}. Another situation in which averaging is undesirable is when learning with reproducing kernels \cite{kivinen2004online}. In such problems, the variable is represented as a linear combination of a kernel evaluated at different points. After $t$ iterations of the subgradient method, the solution is $\sum_{i=1}^{t-1} \alpha_i k(x_i,\cdot)$ where $k:\calH\times\calH\to\mathbb{R}$ is the kernel function. Thus it is necessary to store the $t-1$ points $\{x_i\}$ after $t$ iterations which is infeasible. The key to making the method practical is that for certain objectives the coefficients $\alpha_i$ decay geometrically and the early iterations can be safely ignored. Thus only a small fraction of the last $t$ points are recorded. However, if averaging is used, the earlier coefficients are no longer negligible which compromises the feasibility of the method. Another advantage of our approach over \cite{yang2015rsg} will arise in the next section, where we develop a method for adapting to unknown $c$.



 \section{Double Descending Stairs Stepsize  Method for Unknown $c$}\label{secAdapt}
  \subsection{The Method}
In our method DS-SG (Algorithm \ref{ReSG}), the initial number of inner iterations is
 \begin{eqnarray}
 \label{Kdef}K_1=
   \left\lceil 
   \theta\kappa^2\beta_{ds}^{\frac{1}{2\theta}}\ln\left(2\beta_{ds}\right)
   \Omega_1^{1-\frac{1}{\theta}}
      \right\rceil,
 \end{eqnarray}
 where $\kappa=G/c$. The initial stepsize $\alpha(1)$, given in line \ref{alphaline1}, and the lower bound on $\beta_{ds}$, given in  \eqref{beta_bound}, also depend on $c$.
 If a lower bound for $c$ is known, then using this value in (\ref{alphaline1}), (\ref{beta_bound}), and (\ref{Kdef}) ensures convergence. However in many problems $c$ is unknown.  Furthermore if $c$ is greatly underestimated then this will lead to many more inner iterations and a much smaller initial stepsize than is necessary. For the case where no accurate lower bound for $c$ is known, we propose the following ``doubling trick" which still guarantees essentially the same iteration complexity. The analysis only holds when $\calC$ is bounded. Let the diameter of $\calC$ be $\Omega_{\calC} = \max_{x,x'\in\calC}\|x-x'\|^2$. The basic idea is to repeat DS-SG with a new $c$ which is half the old estimate, which quadruples the number of inner iterations and halves the initial stepsize. In this way it takes only $O(\log_2(\frac{c_1}{c}))$ trial choices for for the error bound constant until it lower bounds the true constant. Furthermore, if the initial estimate $c_1$ is much larger than the true $c$, then the number of inner iterations is relatively small, which is why the overall iteration complexity comes out to be only a factor of $(4/3)$ times larger than that of DS-SG. This means it is advantageous to use a large  overestimate of $c$. In fact one can safely use the initial estimate $c_1 = G\Omega_{\calC}^{\frac{1}{2}-\frac{1}{2\theta}}$. We call the method the ``Doubling trick Descending Stairs" subgradient method (DS2-SG).
 
 \begin{algorithm}
   \caption{Double Descending Stairs subgradient method for unknown $c$ (DS2-SG), $\frac{1}{2}\leq\theta\leq 1$}
   \begin{algorithmic}[1]
   \REQUIRE $\beta_{ds}$, $G$, $M$, $c_1$, $\Omega_{\calC}, x_1$, \emph{stopping criterion}.
   \STATE $l= 1$
   \WHILE{\emph{stopping criterion} not satisfied}
      \STATE $\tilde{x}_l= $DS-SG($\beta_{ds},M,\tilde{x}_{l-1},\Omega_{\calC},G,c_l,\theta$)
      \STATE $c_{l+1}=c_l/2$
      \STATE $l= l+1$
   \ENDWHILE 
   \RETURN $\tilde{x}_{l-1}$
   \end{algorithmic}
   \label{AdReSG}
   \end{algorithm}
   
\begin{theorem}
\label{thmAdapt} Suppose Assumption 3 holds and $1/2\leq\theta\leq 1$. Suppose $\|x-y\|^2\leq \Omega_{\calC}$ for all $x,y\in\calC$ . Let $\kappa_1=G/c_1$.
If $\theta<1$, choose $\beta_{ds}>1$ s.t.
\begin{eqnarray}\label{beta_bound_2_1}
\beta_{ds}&\geq& 
\max\left\{\frac{1}{2}
\left(\frac{\kappa_1^2}{4}\right)^{\frac{\theta}{\theta-1}}\Omega_\calC,
\theta^{-2\theta}\kappa_1^{-4\theta}\Omega_\calC^{2(1-\theta)}
\right\}.
\end{eqnarray}
If $\theta=1$, choose $c_1$ so that $\kappa_1\geq 2$ and choose any $\beta_{ds}>1$.
Fix $\epsilon>0$ and choose 
\begin{eqnarray}\label{Mreq}
M\geq\left\lceil\frac{\ln\frac{\Omega_{\calC}}{\epsilon}}{\ln\beta_{ds}}\right\rceil.
\end{eqnarray}
For the output of Algorithm DS2-SG, if $l\geq L= \max\{0,\lceil\log_2 c_1/c\rceil\}+1$, then $d(\tilde{x}_l,\calX_h)^2\leq\epsilon$. The number of subgradient evaluations is upper bounded by the following quantities (where $\overline{\kappa}=\max\{\kappa,\kappa_1\}$):
\begin{enumerate}
	\item If $\theta=1$:
\begin{eqnarray}
\frac{4}{3}\left(\beta_{ds}^{\frac{1}{2}} \overline{\kappa}^2 \ln(2\beta_{ds})+\left(\frac{\overline{\kappa}}{\kappa_1}\right)^2
+
\log_2\left(\frac{\overline{\kappa}}{\kappa_1}\right)+1
\right)
\left(\frac{\ln\frac{\Omega_{\calC}}{\epsilon}}{\ln\beta_{ds}}+1\right)
\label{adaptIterComp1}
\end{eqnarray}
which simplifies to
\begin{eqnarray*}
O\left(\overline{\kappa}^2\ln\frac{\Omega_{\calC}}{\epsilon}\right)
\end{eqnarray*}
as $\kappa,\kappa_1,\Omega_{\calC}\to\infty$, and $\epsilon\to 0$. 
\item 
If $\frac{1}{2}\leq \theta<1$:
\begin{align}\label{adaptIterComp2}
\frac{4\theta
	\beta_{ds}^{\frac{3}{2\theta}-1}\ln(2\beta_{ds})}
{3(\beta_{ds}^{\frac{1}{\theta}-1}-1)
} 
\overline{\kappa}^2
\epsilon^{1-\frac{1}{\theta}}
+
\left(\frac{4\overline{\kappa}^2}{3\kappa_1^2}
+
\log_2\left(\frac{\overline{\kappa}}{\kappa_1}\right)+2
\right)
\left(
\frac{\ln\frac{\Omega_{\calC}}{\epsilon}}{\ln\beta_{ds}}+1  
\right).
\end{align}
If $\kappa_1$ is chosen large enough so that $\Omega_{\calC} = O(\kappa_1^{\frac{2\theta}{1-\theta}})$,
this simplifies to 
\begin{align}\label{bigO3}
O\left(
\max\{\overline{\kappa}^2,\Omega_{\calC}^{\frac{1}{\theta}-1}\}\epsilon^{1-\frac{1}{\theta}}
\right)
\end{align}
as $\kappa,\kappa_1,\Omega_{\calC}\to\infty$, and $\epsilon\to 0$. 

\end{enumerate}
Note that if $c_1=G\Omega_{\calC}^{\frac{1}{2}-\frac{1}{2\theta}}$, then $\overline{\kappa}=\kappa$. 

\end{theorem}
\begin{proof}
For all $l$ it is clear that since the iterates remain in the constraint set $\calC$,  $d(\tilde{x}_l,\calX_h)^2\leq \Omega_{\calC}$. Now by the choice of $L$, $c_l\leq c$ for all $l\geq L$. Therefore we can apply Theorem \ref{thmRestart} to the iterations within the while loop when $l\geq L$, which implies $d(\tilde{x}_l,\calX_h)^2\leq\epsilon$ for $l\geq L$. Note that, since the R.H.S. of (\ref{beta_bound_2_1}) decreases if you replace $c_1$ with a smaller error bound constant, $\beta_{ds}$ will satisfy (\ref{beta_bound}) with $c_l$ in place of $c_1$ for all $l\geq 2$. 

   We now determine the overall iteration complexity. Let $K_j^l$ for $l=1,2,\ldots, L$ and $j=1,2,\ldots M$ be the number of iterations passed to FixedSG within the $j$th call to FixedSG in DS-SG, during the $l$th loop in DS2-SG. For all such $l$ and $j$, $K_j^l = \lceil\tilde{K}_j^l\rceil$ where $\tilde{K}_j^1=\tilde{K}_j$ defined in \eqref{defKtild}, and $\tilde{K}_j^l = 4^{l-1}\tilde{K}_j^1$. 
  Thus using the fact that $\ceil x< x+1$, the total number of subgradient calls of DS2-SG can be upper bounded as
  \begin{eqnarray}
    \sum_{l=1}^L\sum_{j=1}^M K_j^l
  	&<&
    \sum_{l=1}^L\sum_{j=1}^M\tilde{K}_j^l +  LM
  	\nonumber\\
  	&=&
  	\left(1+4+16+\ldots +4^{L-1}\right)\sum_{j=1}^M\tilde{K}_j^1+  LM
  	\nonumber\\
  	&=&
  	\frac{(4^L-1)}{3}\sum_{j=1}^M\tilde{K}_j^1+LM
  	\nonumber\\
  	&=&
  	\frac{4}{3}
  	\max\left\{\left(\frac{c_1}{c}\right)^2,1\right\}
  	\sum_{j=1}^M\tilde{K}_j^1
  	\nonumber\\
  	&&+(\max\{0,\lceil\log_2 c_1/c\rceil\}+1)
  	\left\lceil\frac{\ln\frac{\Omega_{\calC}}{\epsilon}}{\ln\beta_{ds}}\right\rceil
  	.\label{eqFInal}
  \end{eqnarray}
By noting that 
$$
  	\max\left\{\left(\frac{c_1}{c}\right)^2,1\right\}
  	=
\frac{\overline{\kappa}^2}{\kappa_1^2}
 $$
 and with the aid of \eqref{Th1ResultA} and \eqref{scoreA}, \eqref{eqFInal} reduces to the iteration complexities given in (\ref{adaptIterComp1}) and \eqref{adaptIterComp2}.

 Now $$
 c d(x,\calX_h)^{\frac{1}{\theta}}
 \leq h(x)-h^*\leq \langle g,x-x^*\rangle\leq \|g\|\|x-x^*\|
 $$
 for all $x\in\calC$, $g\in\partial h(x)$, and $x^*\in\calX_h$. If $x^*=\proj_{\calX_h}(x)$ then this implies
 $$
 c d(x,\calH_h)^{\frac{1}{\theta}}\leq Gd(x,\calH_h)\implies c\leq Gd(x,\calH_h)^{1-\frac{1}{\theta}}\quad\forall x\in\calC.
 $$
 Minimizing the R.H.S. yields $c\leq G\Omega_{\calC}^{\frac{1}{2}-\frac{1}{2\theta}}$.
 Therefore the choice $c_1=G\Omega_{\calC}^{\frac{1}{2}-\frac{1}{2\theta}}$ guarantees $\kappa_1\leq\kappa$. 
For $\theta=1$, choosing $c_1=G$ implies $\kappa_1=1$, which violates the requirement: $\kappa_1\geq 2$. Thus one should instead choose $c_1=G/2$. 
\end{proof}
\subsection{Discussion}

    The authors of RSG \cite{yang2015rsg} proposed a variant, R\textsuperscript{2}SG, which can adapt to unknown $c$ when $\theta<1$. It also uses geometrically increasing number of inner iterations, however the initial stepsize remains the same. An advantage of that method is it does not require the constraint set to be bounded. However since their analysis is only valid for $\theta<1$, it cannot be directly applied to important problems such as polyhedral convex optimization, and requires using a surrogate $\theta<1$.
    
    For $\theta=1$, the subgradient methods of \cite[Sec. 2.3]{shor2012minimization} and \cite{goffin1977convergence} choose geometrically decaying stepsizes which depend on the error bound constant $c$. It is plausible that our ``doubling trick" idea can be employed to accelerate these methods when $c$ is unknown, by starting with an estimate for $c$ and repeatedly halving it. This should lead to linear convergence with only a slightly larger iteration complexity than the original methods. Thus our doubling trick can be thought of as a ``meta-acceleration" technique with potentially large scope. 
   
 A drawback of DS2-SG is it does not have an explicit stopping rule. In particular, the number of ``wrapper" iterations, $L$, depends on the true error bound constant $c$, which is unknown. This is also the main drawback of R\textsuperscript{2}SG \cite{yang2015rsg} (along with the fact it cannot be applied when $\theta=1$). As was suggested in \cite{yang2015rsg}, we suggest using an independent stopping criterion. For example on a machine learning problem, one could use the error on a  small validation set as an indication the algorithm has converged. If a lower bound $h_{LB}\leq h^*$ is known, then $c^{-\theta}\left(h(x_k)-h_{LB}\right)^\theta<\sqrt{\epsilon}$ can be used as a stopping criterion. This is because  $d(x_k,\calX_h)\leq c^{-\theta}\left(h(x_k)-h_{LB}\right)^\theta$.  Furthermore since, $c d(x_k,\calX_h)^{\frac{1}{\theta}-1}\leq\|g\|$ for $g\in\partial h(x)$, the norm of the subgradient could be used as a stopping criterion for $\theta<1$. Another possibility is to use the fact that $c d(x_k,\calX_h)\leq \|g\|^\theta \Omega_{\calC}^\theta$. Exploring these stopping criteria is a topic for future work. 
   
   In practice for DS2-SG, we often observe an increase in the objective function value whenever a new trial error bound constant is used resulting in a larger stepsize. It is therefore a good strategy to keep track of the iterate $\tilde{x}_l$ with the smallest objective function value so far. This does not change the overall iteration complexity and only requires storing one additional iterate.  
   

\section{Faster Rates for Decaying Stepsizes for $\theta<1$}\label{sec_sum}
If $\theta<1$, an upper bound for $G$ is known, a lower bound for $c$ is known, and the constraint set is compact, then it is possible to obtain the same iteration complexity as DS-SG using decaying stepsizes. We consider $\theta\geq 1/2$ and $\theta<1/2$ in separate theorems. 

\begin{theorem}\label{OptimalP}\label{ThmOptDecay}
Suppose Assumption 3 holds and $\frac{1}{2}\leq \theta< 1$. Suppose $\|x-y\|^2\leq\Omega_{\calC}$ for all $x,y\in\calC$. 
Choose $c$ small enough (or $G$ large enough) so that 
\begin{eqnarray}
\kappa\label{kappaCond}
\geq
\sqrt{3}\Omega_{\calC}^{\frac{1-\theta}{2\theta}}.
\end{eqnarray}
For the iterates of the subgradient method (\ref{iterSG}), let $\alpha_k = \alpha_1 k^{-p}$ where
\begin{eqnarray}\label{defp}
p = \frac{1}{2(1-\theta)}
\end{eqnarray}
and 
\begin{eqnarray}\label{alph1Choice}
\alpha_1 = \frac{c}{G^2}\left(\frac{\theta\kappa^2}{1-\theta}\right)^p.
\end{eqnarray}
Then, for all $k\geq \lceil\frac{2\theta}{1-\theta}\rceil$
\begin{eqnarray}\label{OptConvRate}
d(x_k,\calX_h)^2 \leq  \left(
\frac{\theta }{1-\theta}
\right)^{\frac{\theta}{1-\theta}} \left(\frac{k}{\kappa^2}\right)^{\frac{\theta}{\theta-1}}.
\end{eqnarray}
\end{theorem}
\begin{proof} 
The recursion describing the subgradient method is, for $k\geq 1$,
\begin{eqnarray}\label{recurting}
e_{k+1}\leq e_k - 2\alpha_k c e_k^{\gamma}+\alpha_k^2 G^2,
\end{eqnarray}
where $e_k=d(x_k,\calX)^2$ and $\gamma=\frac{1}{2\theta}$.
Let $\alpha_k = \alpha_1 k^{-p}$. We wish to prove that if 
$$
p=\frac{\gamma}{2\gamma-1}
$$
and the constant $\alpha_1$ is chosen as in (\ref{alph1Choice}), then
\begin{eqnarray}\label{inductiveHypothesis}
e_k \leq C_e k^{-b}
\end{eqnarray}
 where 
$$b\triangleq\frac{p}{\gamma}=\frac{1}{2\gamma-1},$$ 
for all $k\geq k_0 \triangleq \lceil2 b\rceil$, and $C_e$ is given by $C_e=(\kappa^2 b)^b$.

We will prove this result by induction. The initial condition is
\begin{eqnarray*}
e_{k_0}\leq C_e k_0^{-b}
\end{eqnarray*}
which is implied by
\begin{eqnarray}\label{C_initCond}
\Omega_{\calC}
\leq
C_e k_0^{-b}
\iff
C_e = (\kappa^2 b)^b \geq \Omega_{\calC} k_0^b.
\end{eqnarray}
Since $k_0=\lceil 2b \rceil\leq 2b+1\leq 3b$, this is implied by
\begin{eqnarray*}
\left(
b \kappa^2\right)^{b}
\geq 
\Omega_{\calC}(3b)^b.
\end{eqnarray*}
Dividing by $b^b$ and taking the $b$th root yields 
\begin{eqnarray*}
\kappa^2 \geq 3\Omega_{\calC}^{\frac{1}{b}},
\end{eqnarray*}
which is (\ref{kappaCond}).

Next, assume (\ref{inductiveHypothesis}) is true for some $k\geq k_0$. That is, assume $e_k = a C_e k^{-b}$, where $0\leq a \leq 1$. We will show that this implies $e_{k+1}\leq C_e (k+1)^{-b}$. Substituting $e_k = a C_e k^{-b}$ and $\alpha_k=\alpha_1 k^{-p}$ into the right hand side of (\ref{recurting}) yields
\begin{eqnarray}
e_{k+1}
&\leq&
a C_e k^{-b}-2 \alpha_1 c a^\gamma C_e^\gamma k^{-(p+\gamma b)}+\alpha_1^2 G^2 k^{-2p} 
\nonumber\\\nonumber
&=&
a C_e k^{-b}
+
\left(
\alpha_1^2 G^2
-
2 \alpha_1 c a^\gamma C_e^\gamma
\right)
k^{-2p}
\end{eqnarray}
using the fact that $p+\gamma b=2 p$. Thus we wish to enforce the inequality:
\begin{eqnarray}
\label{TheIn}
a C_e k^{-b}
+
\left(
\alpha_1^2 G^2
-
2 \alpha_1 c a^\gamma C_e^\gamma
\right)
k^{-2p}
\leq
C_e (k+1)^{-b}.
\end{eqnarray}
We need (\ref{TheIn}) to hold for all $a\in[0,1]$. Since $\frac{1}{2}\leq\theta<1$, $\frac{1}{2}<\gamma\leq 1$, therefore the L.H.S. is a convex function of $a$ for $a\geq 0$. Therefore if the inequality holds for $a=0$ and $a=1$, then it holds for all $a\in[0,1]$. 

Consider first, $a=0$. The condition is
\begin{eqnarray*}
\label{alphaCondFirst}
\alpha_1^2 G^2
k^{-2\gamma b}
\leq
C_e (k+1)^{-b}.
\end{eqnarray*}
This is equivalent to
\begin{eqnarray}\label{alphaCond}
\alpha_1 
\leq
G^{-1}C_e^{\frac{1}{2}} k^{\gamma b} (k+1)^{-\frac{b}{2}}.
\end{eqnarray}
Note that $\alpha_1$, given in (\ref{alph1Choice}), can be rewritten as
\begin{eqnarray*}
\alpha_1 = \frac{c C_e^\gamma}{G^2}.
\end{eqnarray*}
 Substituting $\alpha_1$ into (\ref{alphaCond}) yields
\begin{eqnarray*}
\frac{c}{G^2}C_e^\gamma \leq 
G^{-1}C_e^{\frac{1}{2}}k^{\gamma b} (k+1)^{-\frac{b}{2}}
\end{eqnarray*}
which can be rearranged to
\begin{eqnarray}
G\geq
c
C_e^{\gamma-\frac{1}{2}}k^{-\gamma b} (k+1)^{\frac{b}{2}}.\label{love}
\end{eqnarray}
Now 
\begin{eqnarray*}
C_e^{\frac{2\gamma-1}{2}} 
=
\kappa\sqrt{b}.
\end{eqnarray*}
Substituting this into (\ref{love}) yields
\begin{eqnarray}
k^{\gamma b} (k+1)^{-\frac{b}{2}}\geq 
\sqrt{b}.\label{sxx}
\end{eqnarray}
Now
\begin{eqnarray*}
(k+1)^{-\frac{b}{2}}
&=&
k^{-\frac{b}{2}}(1+k^{-1})^{-\frac{b}{2}}
\\
&\geq&
k^{-\frac{b}{2}}
\left(1
-\frac{b}{2}k^{-1}
\right)
\\
&=&
k^{-\frac{b}{2}}
-\frac{b}{2}k^{-\frac{b}{2}-1}.
\end{eqnarray*}
Therefore (\ref{sxx}) is implied by
\begin{eqnarray*}
k^{b(\gamma-\frac{1}{2})}
-
\frac{b}{2}
k^{b(\gamma-\frac{1}{2})-1}
\geq
\sqrt{b}.
\end{eqnarray*}
Now substituting $b=(2\gamma-1)^{-1}$ into the two exponents yields
\begin{eqnarray*}
k^{\frac{1}{2}}-\frac{b}{2}k^{-\frac{1}{2}}\geq\sqrt{b}
\end{eqnarray*}
which is equivalent to
\begin{eqnarray*}
t^2 - \sqrt{b} t - \frac{b}{2}\geq 0 
\end{eqnarray*}
with the substitution $t=\sqrt{k}$. Thus we require
\begin{eqnarray*}
t
\geq
\frac{1+\sqrt{3}}{2}\sqrt{b}
\end{eqnarray*}
which is implied by $k\geq 2 b$. Thus $k\geq \lceil2 b\rceil$ implies (\ref{TheIn}) holds with $a=0$.

Now consider $a=1$ in (\ref{TheIn}). We again simplify (\ref{TheIn}) using
\begin{eqnarray*}
C_e(k+1)^{-b}
=
C_e k^{-b}(1+k^{-1})^{-b}
\geq
C_e k^{-b} - b C_e k^{-(b+1)}.
\end{eqnarray*}
Therefore in the case $a=1$, (\ref{TheIn}) is implied by 
\begin{eqnarray}
\left(
\alpha_1^2 G^2
-
2 \alpha_1 c C_e^\gamma
\right)
k^{-2p}
\leq
 - b C_e k^{-(b+1)}\label{TheIn2}.
\end{eqnarray}
Now $2p=b+1$, therefore (\ref{TheIn2}) is equivalent to
\begin{eqnarray*}
\alpha_1^2 G^2
-
2 \alpha_1 c C_e^\gamma+b C_e
\leq
0
\end{eqnarray*}
for all $k\geq 1$. The L.H.S. is a positive-definite quadratic in $\alpha_1$. Solving it yields the two solutions
\begin{eqnarray*}
\frac{
2c C_e^\gamma \pm \sqrt{4c^2C_e^{2\gamma}-4 G^2 b C_e}
}{2 G^2}.
\end{eqnarray*}
The quadratic has a real solution if
\begin{eqnarray}
4c^2C_e^{2\gamma}-4 G^2 b C_e
\geq 
0
\iff
C_e
\geq
\left(
\kappa^2 b\right)^{b}.
\label{CeLowerBound}
\end{eqnarray}
Thus since $C_e=(\kappa^2 b)^b$, the only valid choice for $\alpha_1$ is
\begin{eqnarray*}\label{AlphaVal}
\alpha_1 = \frac{c C_e^\gamma}{G^2}
\end{eqnarray*}
which corresponds to (\ref{alph1Choice}).
This completes the proof.
\end{proof} 

The convergence rate given in (\ref{OptConvRate}) yields the following iteration complexity: The subgradient method with this stepsize yields a point such that $d(x_k,\calX_h)^2\leq\epsilon$ for all
\begin{eqnarray*}
k\geq \frac{2\theta}{1-\theta} \max\{\kappa^2,3\Omega_{\calC}^{\frac{1}{\theta}-1}\} \epsilon^{1-\frac{1}{\theta}}.
\end{eqnarray*}
This is equal (up to constants) to the iteration complexity derived for DS-SG in Theorem \ref{thmRestart}. The main drawback versus DS-SG is that the analysis only holds for a bounded constraint set. It is also trivial to embed this stepsize into the ``doubling" framework used in DS2-SG so that one does not need a lower bound for $c$. Since the analysis is the same as given in Theorem \ref{thmAdapt}, we omit the details. The proof of Theorem \ref{ThmOptDecay} is inspired by \cite{goffin1977convergence} which considered geometrically decaying stepsizes when $\theta=1$. Theorem \ref{ThmOptDecay} is a natural extension of \cite{goffin1977convergence} to $\theta<1$. 

The optimal stepsize given in Theorem \ref{ThmOptDecay} requires knowledge of $G$, $c$, and $\Omega_{\calC}$ in order to set $\alpha_1$. In the longer version of this paper \cite{johnstone2017faster} we show that the stepsizes $\alpha_k=\alpha_1 k^{-p}$ with $p<1$ are convergent for any $\alpha_1>0$ when $\theta\geq 1/2$. 
 
We can obtain the same rate for the choice of $\alpha_1$ and $p$ in Theorem \ref{ThmOptDecay} when $\theta<1/2$. In this case, the convergence rate holds for all $k\geq 2$ under a slightly different condition on $\kappa$. 

\begin{theorem}\label{OptimalP2}
Suppose Assumption 3 holds and $0< \theta< \frac{1}{2}$. 
Suppose $\|x-y\|^2\leq\Omega_{\calC}$ for all $x,y\in\calC$. 
Choose $c$ small enough (or $G$ large enough) so that 
\begin{eqnarray}
\kappa^2\label{secondKappaCond}
\geq
\frac{2(1-\theta)}{\theta}\Omega_\calC^{\frac{1-\theta}{\theta}}.
\end{eqnarray}
For the iterates of the subgradient method (\ref{iterSG}), let $\alpha_k = \alpha_1 k^{-p}$ 
where $p$
and $\alpha_1$ are defined in \eqref{defp}  and (\ref{alph1Choice}).
Then, for all $k\geq 2$, $d(x_k,\calX)^2$ satisfies (\ref{OptConvRate}).
\end{theorem}
\begin{proof}
Recall $\gamma=1/(2\theta)$ and note that $\gamma> 1$ since $\theta< 1/2$. Recall 
$$
b=\frac{1}{2\gamma-1}\leq 1\text{ and } p=\gamma b.
$$
As with the proof of Theorem \ref{OptimalP}, this will be a proof by induction. We wish to prove that $e_k\leq C_e k^{-b}$ for all $k\geq 2$ for the constant defined as 
$
C_e = \left(\kappa^2 b\right)^b.
$
 The initial condition is 
$
e_{2}\leq C_e 2^{-b}
$
which is implied by
$
C_e\geq\Omega_\calC 2^b\label{C_initCond2}.
$
This in turn is implied by (\ref{secondKappaCond}).

Now we assume $e_k=aC_e k^{-b}$ for some $k\geq 2$ and $a\in[0,1]$ and will show that $e_{k+1}\leq C_e(k+1)^{-b}$. Using  the inductive assumption in the main recursion (\ref{recurting}) yields the following inequality, which we would like to enforce for all $a\in[0,1]$:
\begin{eqnarray}\label{astar}
e_{k+1}&\leq& a C_e k^{-b}
+
\left(
\alpha_1^2 G^2
-
2 \alpha_1 c a^\gamma C_e^\gamma
\right)
k^{-2p}
\nonumber\\
&\leq&
C_e (k+1)^{-b},
\end{eqnarray}
where we once again used the fact that $p+\gamma b=2 p$. 
We require (\ref{astar}) to hold for all $a\in[0,1]$. The L.H.S. is concave in $a$ (since $\gamma> 1$), so we will compute the maximizer w.r.t. $a$. Let $D_1 = \alpha_1^2 G^2 k^{-2p}$, $D_2 = C_e k^{-b}$, and $D_3= 2\alpha_1 c C_e^\gamma k ^{-2\gamma b}$. Then let
\begin{eqnarray*}
f(a)=D_1 + D_2 a-D_3 a^\gamma
\end{eqnarray*}
which is the L.H.S. of (\ref{astar}).
Let $a_*$ be the solution to 
$$
0 = f'(a_*) = D_2 - \gamma D_3 a_*^{\gamma - 1},
$$ 
which implies
\begin{eqnarray*}
a_* &=& \left(\frac{D_2}{\gamma D_3}\right)^{\frac{1}{\gamma-1}}
\\
&=&
C_e^{-1}
(2\alpha_1\gamma c)
^{\frac{1}{1-\gamma}}
k^{\frac{1}{\gamma-1}}
= 
C_e^{-1}D_4
\alpha_1^{\frac{1}{1-\gamma}}
k^{\frac{1}{\gamma-1}}
\end{eqnarray*}
where $D_4 = (2\gamma c)^{\frac{1}{1-\gamma}}$. But recall that $a\in[0,1]$ therefore the maximizer of $f(a)$ in $[0,1]$ is given by
\begin{eqnarray*}
\min\left\{1,C_e^{-1}D_4
\alpha_1^{\frac{1}{1-\gamma}}
k^{\frac{1}{\gamma-1}}\right\}.
\end{eqnarray*}
Thus if 
\begin{eqnarray}\label{secondKcond}
k\geq (C_e D_4^{-1})^{\gamma-1}\alpha_1
\end{eqnarray}
then the maximizer in $[0,1]$ is equal to $1$. Substituting the values for $\alpha_1$ and $C_e$ into (\ref{secondKcond}) yields
\begin{eqnarray*}
k&\geq& (C_e D_4^{-1})^{\gamma-1}\frac{c}{G^2}C_e^\gamma
=
\frac{2\gamma}{2\gamma-1}.
\end{eqnarray*}
Since $\gamma> 1$ this is implied by $k\geq 2$. Thus we only need to consider $a=1$ in (\ref{astar}).


The analysis with $a=1$ substituted into (\ref{astar}) was carried out in the proof of Theorem \ref{ThmOptDecay}. Recall that for this choice of stepsize and constant, the inequality (\ref{astar}) is satisfied with $a=1$ for all $k\geq 1$, which completes the proof.

\end{proof}

\section{Convergence Rates for Classical Nonsummable Stepsizes}\label{Sec_decay}
We now turn our attention to nonsummable but square summable stepsize sequences for the subgradient method under HEB. These stepsizes are used frequently for the stochastic and deterministic subgradient method, however their behavior under HEB has not been studied in detail with the exception of \cite{lim2011convergence,supittayapornpong2016staggered}. We will see that these nonsummable stepsizes are slower than the ``descending stairs" stepsizes and summable stepsizes when $\theta> 1/2$. However, in this case the nonsummable stepsizes have the advantage that they do not require $G$, $c$, and $\Omega_1$. We will first state and discuss our results. The proofs are in Section \ref{secProofDecay}. 
\subsection{Results for $\theta\in(0,\frac{1}{2})$}

\begin{theorem}\label{thmDimSum}
Suppose Assumption 3 holds and $0<\theta<1/2$. Let $\alpha_k = \alpha_1 k^{-p}$. Let
\begin{eqnarray}
C_1 &\triangleq& \label{C1Def}
2^{2p\theta+1}
\left(
\left(\frac{\alpha_1 G^2}{c}\right)^{2\theta}
+
\alpha_1^2 G^2
\right)
\\\nonumber
C_2
&\triangleq&
\left(\frac{\alpha_1(1-2\theta)}{2\theta(1-p)}\right)^{\frac{2\theta}{2\theta-1}}.
\end{eqnarray}
Then if
\begin{eqnarray}\label{dcond}
\frac{1}{2(1-\theta)}\leq p\leq 1
\end{eqnarray}
and $\alpha_1$ is chosen so that
\begin{eqnarray}
C_1
&\leq&
\left(\frac{2\theta(1-p)}{\alpha_1(1-2\theta)}\right)^{\frac{2\theta}{1-2\theta}}
(k_0+1)
^{\frac{2\theta(2p(1-\theta)-1)}{1-2\theta}}\label{cc},
\\
\alpha_1
&\leq&
\frac{2\theta(1-p)d(x_1,\calX_h)^{\frac{2\theta-1}{\theta}}}{1-2\theta}
,
\label{cc2}
\end{eqnarray}
then for all $k\geq k_0$
\begin{eqnarray}\label{lowTheta1}
d(x_k,\calX_h)^2
\leq
\max\{C_1,C_2\}
\max\left\{
k^{-2p\theta},
k^{\frac{2\theta(1-p)}{2\theta-1}}
\right\}.
\end{eqnarray}

\end{theorem}
\begin{proof}
Sec. \ref{secProofDecay}.
\end{proof}

In the following corollary we give the optimal choice for $p$ that makes the two arguments to the max function in (\ref{lowTheta1}) equal. 
\begin{corollary}\label{CorDimSum}
In the setting of Theorem \ref{thmDimSum} with $0<\theta<\frac{1}{2}$ and $C_1$ defined in (\ref{C1Def}),  if $p=\frac{1}{2(1-\theta)}$, and $\alpha_1$ is chosen so that (\ref{cc2}) holds and
\begin{eqnarray}\label{SecondAConstant}
\alpha_1^{\frac{2\theta}{1-2\theta}}
C_1\leq \left(\frac{\theta}{1-\theta}\right)^{\frac{2\theta}{1-2\theta}}
\end{eqnarray}
then for all $k\geq 1$
\begin{eqnarray*}\label{lowTheta2}
d(x_k,\calX_h)^2\leq 
\alpha_1^{\frac{2\theta}{2\theta-1}}
\left(\frac{\theta}{1-\theta}\right)^{\frac{2\theta}{1-2\theta}} 
k^{\frac{-\theta}{1-\theta}}. 
\end{eqnarray*}
If $\alpha_1$ is chosen so that (\ref{SecondAConstant}) is satisfied with equality, then 
\begin{eqnarray*}\label{lowTheta3}
d(x_k,\calX_h)^2\leq 
C_1
k^{\frac{-\theta}{1-\theta}}. 
\end{eqnarray*}

\end{corollary}
\begin{proof}
Sec. \ref{secProofDecay}.
\end{proof}

Our derived convergence rate $O(k^{\frac{-\theta}{1-\theta}})$ is faster than the naive application of the classical $O(1/\sqrt{k})$ function value convergence rate, which with the use of HEB results in a rate $d(\hat{x}_k,\calX_h)^2=O(k^{-\theta})$ at the averaged point $\hat{x}_k=\sum\alpha_k x_k/\sum\alpha_k$. Furthermore our result is nonergodic (no averaging is required). 
Thus we see that for $\theta<1/2$ decaying polynomial stepsize sequences can achieve the same convergence rate as RSG \cite{yang2015rsg} and the constant stepsize we derived in Theorem \ref{ThmFixIterComp}. 

\subsection{Results for $\theta\in[\frac{1}{2},1]$}
We now consider nonsummable stepsizes for $\theta\geq 1/2$. The primary advantage of the following stepsize is that it does not require knowledge of $G,c$, or $d(x_1,\calX)^2$. 
\begin{theorem}\label{ThmLargeTheta}
Suppose Assumption 3 holds and $1/2\leq \theta\leq 1$.
Suppose $\alpha_k = \alpha_1 k^{-p}$ for some $p\in(0,1)$ and $\alpha_1>0$. Let $C_1$ be as defined in (\ref{C1Def}), 
\begin{eqnarray*}
C_3 &\triangleq& C_1^{\frac{1+2p(\theta-1)}{1-p}}
\left(
\frac{\alpha_1(1-2^{p-1}) c e}{4p\theta}
\right)^{-\frac{2 p\theta}{1-p}}
\\
C_4 &\triangleq& 16 \left(\frac{8\theta C_1}{\alpha_1 c e}\right)^{2\theta}
\\
C_5 &\triangleq& d(x_1,\calX_h)^{\frac{2+4p(\theta-1)}{1-p}}\left(
\frac{\alpha_1 c e}{4p\theta}
\right)^{-\frac{2 p\theta}{1-p}}.
\end{eqnarray*}
Then for all $k\geq 4$
\begin{eqnarray}\label{BigThetaResult}
d(x_k,\calX_h)^2\leq 4\max\{C_1,C_3,C_4,C_5\}k^{-2p\theta}.
\end{eqnarray}
\end{theorem}
\begin{proof}
Sec. \ref{secProofDecay}.
\end{proof}

Once again this improves on the known classical \emph{ergodic}  convergence rate of $O(k^{-\theta})$. As $p\to1$ the method can get arbitrarily close to the best rate $O(k^{-2\theta})$, however $p=1$ is not covered by our analysis other than the special case $\theta=\frac{1}{2}$ discussed in Theorem \ref{thmD1} and Proposition \ref{PropThta05} below. The decaying stepsize does not require knowledge of $\theta$, $c$, $G$, $h^*$, or $d(x_1,\calX_h)$ to set the parameters $\alpha_1$ and $p$. 
The result holds for arbitrary $\alpha_1>0$ and $p\in(0,1)$. Nevertheless, the constants are affected by the choice of $\alpha_1$ and $p$ as well as practical performance.

The convergence rate for the decaying stepsizes is much slower than DS-SG, the summable stepsizes in Sec. \ref{sec_sum}, and RSG \cite{yang2015rsg}. These methods obtain the rate $O\left(k^{\frac{\theta}{\theta-1}}\right)$ for $\theta>1/2$. 


The case $\theta=1$ in Theorem \ref{ThmLargeTheta} can be compared with the main result of \cite{lim2011convergence} which also proves $O(1/k^2)$ rate of convergence for $d(x_k,\calX_h)^2$. A difference is their result only holds for sufficiently large $k$. They also assume the function satisfies the quadratic growth condition (i.e. $\theta=1/2$ error bound) globally. For problems where $\calC$ is compact, this does not matter, since QG is implied by WS on a compact set. An advantage of \cite{lim2011convergence} is that it holds for stochastic gradient descent.

\subsection{Results for $\theta=\frac{1}{2}$}
For the special case of $\theta=\frac{1}{2}$ our analysis extends to the choice $p=1$.
\begin{theorem}\label{thmD1}
Suppose Assumption 3 holds and $\theta=1/2$. 
Suppose $\alpha_k = \alpha_1 k^{-1}$ and
\begin{eqnarray*}
\alpha_1\leq\frac{1}{c}.
\end{eqnarray*}
Then for $k\geq 1$
\begin{eqnarray}\label{ThmD1Result}
d(x_k,\calX_h)^2
\leq \max\left\{\frac{2 \alpha_1 G^2}{c},d(x_1,\calX_h)^2\right\} k^{-c \alpha_1}.
\end{eqnarray}
\end{theorem}
\begin{proof}
Sec. \ref{secProofDecay}.
\end{proof}

Strongly convex functions with strong convexity parameter $\mu_{sc}$ satisfy the error bound with $\theta=\frac{1}{2}$ and $c=\frac{\mu_{sc}}{2}$. In this case $C_1=\frac{8 G^2}{c^2}$. Thus, for the choice $\alpha_1=\frac{2}{\mu_{sc}}$ we have proved that
\begin{eqnarray*}
d(x_k,\calX_h)^2\leq \frac{1}{k}\max\left\{d(x_1,\calX_h)^2,\frac{32 G^2}{\mu_{sc}^2}\right\}.
\end{eqnarray*}
This result can be compared with several papers. The result \cite[Theorem 6.2]{bubeck2015convex} finds an $O(1/k)$ convergence rate for $h(\hat{x}_k)-h^*$ for a particular averaged point $\hat{x}_k$ under strong convexity. This, combined with HEB implies an $O(1/k)$ rate for $d(\hat{x}_k,\calX_h)^2$.  The work \cite[Thm 1]{nedic2014stochastic} obtained a nonergodic $O(1/k)$ rate for $d(x_k,\calX_h)^2$ in stochastic mirror descent under strong convexity for a similar stepsize sequence to Theorem \ref{thmD1}. The result \cite[Prop. 2.8]{nedic2001convergence} provides convergence rates for the (incremental) subgradient method with stepsize $\alpha_k=\alpha_1k^{-1}$ for all values of $\alpha_1$ under QG. This is more general than Theorem \ref{thmD1} as they cover the case where $\alpha_1>1/c$. However, for $\alpha_1=1/c$, \cite[Prop. 2.8]{nedic2001convergence} only proves $O(\log k/k)$ convergence whereas Theorem \ref{thmD1} implies $O(1/k)$ convergence. 
The result of \cite[Eq. (2.9)]{nemirovski2009robust} says that for strongly convex functions with parameter $\mu_{sc}$, the subgradient method achieves a nonergodic $O(1/k)$ convergence so long as $\alpha_1>\frac{1}{2\mu_{sc}}$.  In contrast we do not require strong convexity but only the weaker error bound. The result can also be compared to \cite[Thm. 4]{karimi2016linear} which proved an $O(1/k)$ rate for the objective function gap under QG. However they additionally require Lipschitz smoothness. Both \cite{nemirovski2009robust} and \cite{karimi2016linear} considered the stochastic subgradient method. 

We also provide another choice of stepsize which guarantees a convergence rate of $O(1/k)$ for $d(x_k,\calX_h)^2$ in the case where $\theta=\frac{1}{2}$. This proof is a direct adaptation of \cite[Thm. 4]{karimi2016linear}. Unlike \cite[Thm. 4]{karimi2016linear}, it does not require smoothness of the objective.
\begin{proposition}\label{PropThta05}
In the setting of Theorem \ref{thmD1}, consider the subgradient method with 
\begin{eqnarray*}
\alpha_k =\frac{2k+1}{2c(k+1)^2}.
\end{eqnarray*}
Then for all $k$
\begin{eqnarray*}
d(x_{k+1},\calX_h)^2 
\leq
\frac{d(x_1,\calX_h)^2}{(k+1)^2}
+
\frac{G^2}{c^2(k+1)}.
\end{eqnarray*}
\end{proposition}
\begin{proof}
Sec. \ref{secProofDecay}.
\end{proof}

\section{Numerical Experiments}\label{sec_numerical}
In this section we present simulations to demonstrate some of the theoretical findings in this manuscript. We consider two examples satisfying HEB($c,\theta$) with $\theta=1$ to test our proposed descending stairs stepsize choice in DS-SG and our ``double descending stairs" method for unknown $c$, DS2-SG. 
\subsection{Least-Absolute Deviations Regression}
Consider the following problem:
\begin{eqnarray}\label{prob_l1_l1}
\min_x \|Ex-b\|_1:\quad\|x\|_1\leq\tau.
\end{eqnarray}
This objective function is often used in regression problems and in machine learning \cite{hastie2009elements,wang2006regularized,wang2013l1,gao2010asymptotic}. 
Besides the subgradient techniques considered in this manuscript, there are a few other methods which can tackle Prob. (\ref{prob_l1_l1}). The problem can be written as a linear program and solved via any LP solver. A popular option is an interior point method. These are second order methods that rely on computing second order information and solving potentially large linear systems at each iteration. In general they are not competitive with subgradient methods on large scale problems. Simplex methods \cite{barrodale1973improved} are another option. While their typical performance is good, these methods have exponential computational complexity in the worst case. The alternating direction method of multipliers (ADMM) is another approach to solving Prob. (\ref{prob_l1_l1}), however it involves solving a quadratic program at each iteration, placing it in the same complexity class as the interior point methods \cite{eckstein2012augmented}. The primal-dual splitting method of \cite{chambolle2011first} is a first-order method which can tackle Prob. (\ref{prob_l1_l1}). 
The main drawback of the method is that one must know the largest singular value of $E$ in order to choose the stepsizes correctly. As such, it is not directly comparable with the subgradient methods developed in this manuscript which do not require this information. The paper \cite{wang2006regularized} introduces a method for solving Prob. (\ref{prob_l1_l1}) which is similar to the LARS method for solving the LASSO \cite{efron2004least}. The method solves Prob. (\ref{prob_l1_l1}) for an increasing sequence of $\tau$. At every iteration it solves a linear system, using the previous solution in a smart way. However, as far as we are aware, the iteration complexity of this method is unknown. Edgeworth's algorithm is a coordinate descent method for Prob. (\ref{prob_l1_l1}) which has shown promising empirical performance \cite{wu2008coordinate}. However unlike the subgradient methods considered here, the method is not guaranteed to converge to a minimizer. In fact specific examples exist where Edgeworth's algorithm converges to a non-optimal point \cite{li2004maximum}. 

Problem (\ref{prob_l1_l1}) is a polyhedral optimization problem therefore HEB($c,\theta$) is satisfied for all $x$ with $\theta=1$ \cite{yang2015rsg}. However, it is not easy to compute $c$. Note that the constraint set is compact thus DS2-SG is applicable. Projection onto the $\ell_1$ ball can be done in linear time in expectation via the method of \cite{duchi2008efficient}. 

To test the subgradient methods we first consider a random instance of Problem (\ref{prob_l1_l1}). We set $m=100$ and $n=50$ and construct $E$ of size $m\times n$ with i.i.d. $\mathcal{N}(0,1)$ entries. We construct $b$ of size $m\times 1$ with  i.i.d. $\mathcal{N}(0,1)$ entries. We set $\tau=1$. All tested algorithms were initialized to the same point.

To start we test the convergence rates predicted by Theorem \ref{ThmLargeTheta} for decaying stepsizes. We consider two stepsizes $\alpha_k = 0.1 k^{-0.99}$, and $\alpha_k = 0.01 k^{-0.5}$, where the constants were tuned to achieve good performance. In Fig. \ref{fig_decay} we plot the log of $d(x_k,\calX_h)^2$ versus $\log_{10} k$, where $k$ is the number of iterations. An optimal solution $x^*$ is estimated by running DS-SG until it converges to within numerical precision. Looking at the figure it appears that for $k>1000$ the convergence rates are as predicted in Theorem \ref{ThmLargeTheta}. Specifically for the first parameter choice, $d(x_k,\calX_h)^2 \approx O(k^{-1.98})$ and for the second $d(x_k,\calX_h)^2 \approx O(k^{-1})$.

Next we test the performance of DS-SG, RSG \cite{yang2015rsg}, and Shor's method of \cite[Sec. 2.3]{shor2012minimization} (which is very similar to Goffin's stepsize \cite{goffin1977convergence}), alongside the two decaying stepsizes discussed in Fig. \ref{fig_decay}.  For DS-SG we used $\beta_{ds}=4$, $\epsilon=10^{-5}$, $\Omega_{\calC}=4\tau^2$, and $G=\sum_{i=1}^n\|E_i\|$ where $E_i$ is the $i$th column of $E$. For the other methods we chose the parameters in the way suggested by the authors.
 Since $c$ is difficult to estimate, we tuned it to get the best performance in each algorithm (see below for our approach, DS2-SG, which does not need $c$). For DS-SG, RSG, and Shor's algorithm, these were $c=22,15$, and $11$ respectively. 

The log of $d(x_k,\calX_h)^2$ for each of these algorithms is plotted in Fig. \ref{fig_compAll} versus the number $k$ of subgradient evaluations.
Fig. \ref{fig_compAll} confirms that DS-SG has a linear convergence rate, verifying Theorem \ref{thmRestart}. It's performance is very similar to Shor's method. While RSG does appear to obtain linear convergence, it's rate is slower than DS-SG and Shor's method. 

\begin{figure}
\centering
\includegraphics[width=3in]{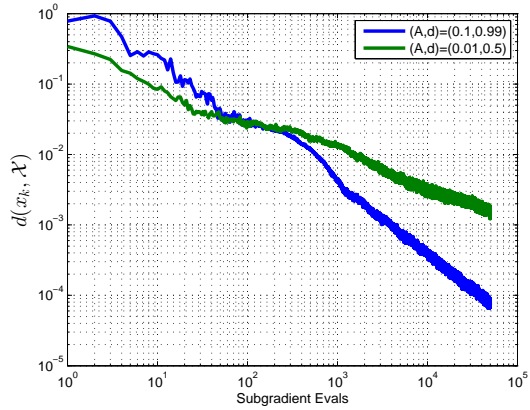}
\caption{Problem (\ref{prob_l1_l1}): Log of square distance to the (unique) solution vs log of number of subgradient evaluations for two decaying stepsizes.}
\label{fig_decay}
\end{figure}

\begin{figure}
\centering
\includegraphics[width=3in]{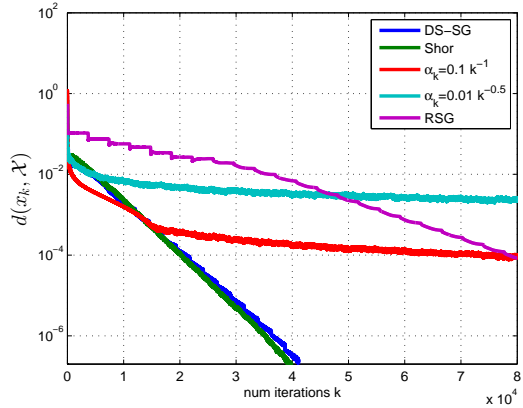}

\caption{Problem (\ref{prob_l1_l1}): Log of square distance to the (unique) solution vs number of subgradient evaluations for DS-SG, RSG, and two decaying stepsizes.}
\label{fig_compAll}
\end{figure}

\begin{figure}
\centering
\includegraphics[width=3in]{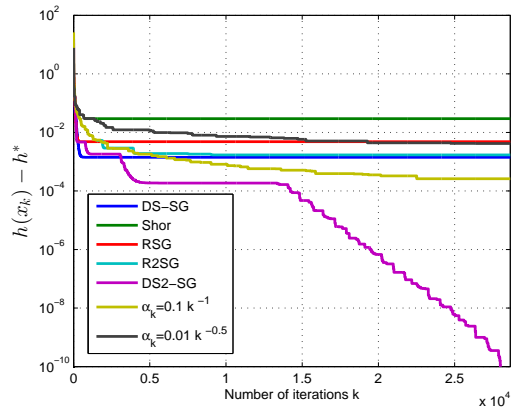}
\caption{Problem (\ref{prob_l1_l1}): Log of $h(x)-h^*$  vs number of subgradient evaluations for DS-SG, RSG, and Shor's method all with $c=100$, R\textsuperscript{2}SG, DS2-SG with the initial $c_1=G=160$, and two decaying stepsizes.}
\label{fig_ad}
\end{figure}

As was mentioned we had to tune $c$ to get good performance of DS-SG, RSG, and Shor's method. We now compare these three methods with our proposed 'doubling trick' variant DS2-SG, which does not need the value of $c$. We also compare with the method R\textsuperscript{2}SG proposed in \cite{yang2015rsg}. Note that this method only works for $\theta<1$ so following the advice of \cite{yang2015rsg}, we use the approximate value of $\hat{\theta}=0.8$, which was chosen because it performed well.  We initialize DS2-SG with the same parameters as DS-SG but with $c_1 = G = 160$. To demonstrate the effect of poorly chosen $c$ in DS-SG, RSG, and Shor's method, we set $c=100$ for all these methods (recall the tuned values were smaller). The results are given in Fig. \ref{fig_ad}. We compare function values and for each algorithm we keep track of the iterate with the smallest function value so far. We see that DS-SG, RSG, and Shor's method converge to suboptimal solutions due to the incorrect value of $c$. However DS2-SG finds the correct solution to within an objective function error of $10^{-10}$. R\textsuperscript{2}SG has slower convergence, which is not surprising since it is not guaranteed to obtain linear convergence when $\theta=1$. It is also encouraging that DS2-SG is faster than the decaying stepsizes $\alpha_k=O( k^{-1})$ and $\alpha_k=O(k^{-0.5})$, since this choice also does not require knowledge of $c$.

\subsection{Least-Absolute Deviations Regression on the ``space.ga" Dataset}

We also apply Prob. (\ref{prob_l1_l1}) on a real dataset. We use the normalized space.ga dataset downloaded from the \emph{libsvm} website.\footnote{\url{https://www.csie.ntu.edu.tw/~cjlin/libsvmtools/datasets/}.} We use a subset of the dataset with $m=100$ and $n=6$, and set $\tau=5$.

Since $c$ is unknown, we compare subgradient methods which do not require it. Thus we compare two decaying stepsizes, $\alpha_k= k^{-1}$ and $\alpha_k=0.1 k^{-0.5}$, and DS2-SG. Note that  R\textsuperscript{2}SG also does not require $c$ but we could not tune it to be competitive on this problem. For DS2-SG, we estimate $G=\sum_{i=1}^n \|E_i\|$ and $\Omega_{\calC}=4\tau^2$ as in the synthetic experiment. We use $\beta_{ds}=2$ and $\epsilon=10^{-12}$. The objective function vs iteration-number is plotted in Fig. \ref{fig_realRegress}. One can see that the decaying stepsizes are faster than DS2-SG in the early iterations but DS2-SG is much faster in the later iterations. The decaying stepsizes were highly sensitive to the choice of $\alpha_1$ which had to be tuned. On the other hand DS-SG was effected by the choice of $\beta_{ds}$. Smaller values of $\beta_{ds}$ lead to better performance early-on, while larger values give better convergence in the latter iterations. In general $\beta_{ds}\in[1.5,4]$ worked well in all of our experiments. 

\begin{figure}
\centering
\includegraphics[width=3in]{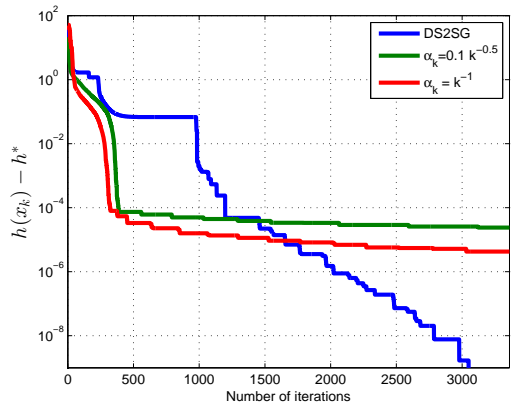}
\caption{Problem (\ref{prob_l1_l1}) applied to space.ga dataset: Log of $h(x)-h^*$  vs number of subgradient evaluations for DS-SG, $\alpha_k=k^{-1}$, and $\alpha_k=0.1 k^{-0.5}$.}
\label{fig_realRegress}
\end{figure}

\subsection{Sparse SVM}
The $\ell_1$-regularized Support Vector Machine (SVM) Problem \cite{zhu20031} is
\begin{eqnarray*}
	\min_{x\in\mathbb{R}^n}\sum_{i=1}^m \max\{0,1- y_i c_i^\top x\}  +\rho\|x\|_1
\end{eqnarray*}
for a dataset $\{c_i,y_i\}_{i=1}^m$ with $c_i\in\mathbb{R}^n$ and $y_i\in\{\pm 1\}$. We will consider the equivalent constrained version
\begin{eqnarray}\label{probSVM}
\min_{x\in\mathbb{R}^n}\sum_{i=1}^m \max\{0,1- y_i c_i^\top x\} : \|x\|_1\leq\tau.
\end{eqnarray}
Since the objective function is polyhedral it satisfies HEB with $\theta=1$ for some unknown $c>0$. Once again since $c$ is unknown, we only consider DS2-SG, R\textsuperscript{2}SG \cite{yang2015rsg}, and the following decaying stepsizes: $\alpha_k=0.1 k^{-1}$ and $\alpha_k=0.01 k^{-0.5}$, where the constants $0.1$ and $0.01$ were tuned to give fast convergence. R\textsuperscript{2}SG only works for $\theta<1$ so cannot be directly applied to this problem. Instead we selected $\hat{\theta}<1$ which gave the fastest convergence. Surprisingly, $\hat{\theta}=0.5$ performed the best even though one might expect $\hat{\theta}\approx 1$ to perform better. For DS2-SG we initialize with $c_1=G$ where $G=\sum_{i=1}^m\|c_i\|$. We used $\beta_{ds}=2$, $\epsilon=10^{-5}$, and $\Omega_{\calC}=4\tau^2$. All four algorithms had the same starting point. 

A random instance of Prob. (\ref{probSVM}) was generated as follows: $n=50$, $m=100$, the entries of $c_i$ are drawn from $\mathcal{N}(0,1)$, the $y_i=\pm 1$ with equal probability, and $\tau=2$. The results are plotted in Fig. \ref{figSVM}. We see that our proposal, DS2-SG, outperforms the others. 

\subsection{Sparse SVM on the ``glass.scale" Dataset}
To test Prob. (\ref{probSVM}) on real data, we download the \emph{glass.scale} dataset from the libsvm website. For this dataset, $n=9$ and $m=214$. There are $6$ different labels so we group labels ``1", ``2", and ``3" together into class: $y=-1$, and labels ``5", ``6", and ``7" into class: $y=1$. We solve Prob. (\ref{probSVM})  with $\tau=2$. 

Again we compare the subgradient methods which do not require $c$, namely DS2-SG, R\textsuperscript{2}SG, and two decaying stepsizes. For DS2-SG we use the same parameters as in the synthetic experiment, except $\beta_{ds}=4$ and $\epsilon=10^{-8}$. The  objective function vs iteration-number is plotted in Figure \ref{fig_realSVM}. Once again we see that DS2-SG outperforms the two decaying stepsizes as well as R\textsuperscript{2}SG.

\begin{figure}[h!]
\centering
\includegraphics[width=3in]{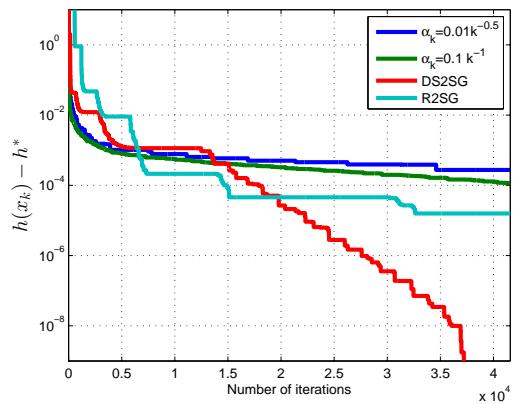}
\caption{Problem (\ref{probSVM}) with randomly generated data: Log of $h(x)-h^*$  vs number of subgradient evaluations for DS2-SG,   R\textsuperscript{2}SG, and two decaying stepsizes.}
\label{figSVM}
\end{figure}

\begin{figure}[h!]
\centering
\includegraphics[width=3in]{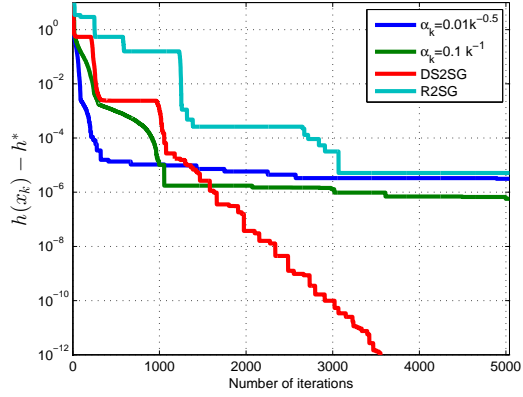}
\caption{Problem (\ref{probSVM}) for the ``glass.scale" dataset: Log of $h(x)-h^*$  vs number of subgradient evaluations for DS2-SG,   R\textsuperscript{2}SG, and two decaying stepsizes.}
\label{fig_realSVM}
\end{figure}

\section{Extensions}\label{secExtend}
As previously mentioned, the key recursion (\ref{KeyRecursion}) can also be derived in the following situations: 1) when a small amount of noise is added to the subgradient, 2) for the incremental subgradient method, 3) under a more general condition than HEB, introduced by Goffin \cite{goffin1977convergence},  4) for the proximal subgradient method, and 5) for relaxed versions of the subgradient method. We now discuss the first three of these in more detail. 

\subsection{Deterministic Noise in the Subgradient when $\theta=1$}
    For the weakly sharp case ($\theta=1$), the subgradient method exhibits resilience to bounded noise. This has been observed in \cite{nedic2010effect,poljak1978nonlinear}. Suppose that at each iteration we have access to a noisy subgradient:
    \begin{eqnarray*}
    \tilde{g}_k = g_k+r_k:g_k\in\partial h(x_k),\|r_k\|\leq R
    \end{eqnarray*}
    and as before the method iterates for all $k\geq 0$
    \begin{eqnarray*}
     x_{k+1}=P_{\calC}(x_k - \alpha_k\tilde{g}_k).
    \end{eqnarray*}
One can repeat the analysis of Sec. \ref{secMy} to show
     \begin{eqnarray*}
     d(x_{k+1},\calX_h)^2
     &\leq&
     d(x_k,\calX_h)^2
     -2\alpha_k  d(x_k,\calX_h)(c-R)
     +2\alpha_k^2 (R^2+G^2).
     \end{eqnarray*}
     We see that this is exactly the same recursion as (\ref{ABiggy}) with the error bound constant $c$ replaced by $c-R$, and $G^2$ replaced by $2(G^2+R^2)$. Thus, if $R<c$, all of the results presented throughout for $\theta=1$ hold with a new error bound constant $\tilde{c}=c-R$, and bound on the subgradients $\tilde{G}^2 = 2(G^2+R^2)$. In particular this refers to Theorems \ref{ThmFix}, \ref{ThmFixIterComp}, \ref{thmRestart}, \ref{thmAdapt}, and \ref{ThmLargeTheta}.
 
 \subsection{Incremental Subgradient Methods}
 Suppose $h(x)=\sum_{i=1}^m h_i(x)$. Such objective functions which are a finite sum of terms often arise in machine learning in the guise of \emph{empirical risk minimization} \cite{hastie2009elements}. For such problems the \emph{incremental} subgradient method can be used \cite{nedic2001convergence}. This method proceeds by computing the subgradient with respect to each individual function $h_i$ in a fixed order. More precisely the method proceeds for $k\geq 1$ with $x_1\in\calC$ as
 \begin{eqnarray}\label{increment1}
 x_{k+1}&=&\psi_{m,k}
 \\\label{increment2}
 \psi_{i,k}&=& P_\calC(\psi_{i-1,k}-\alpha_k g_{i,k}),g_{i,k}\in\partial h_i(\psi_{i-1,k}),\,\,\, i=1,\ldots,m
 \\\label{increment3}
 \psi_{0,k}&=&x_k.
 \end{eqnarray}
  This method has been analyzed extensively in \cite{nedic2001convergence}.
  \begin{proposition}[\cite{nedic2001convergence}]
Suppose Assumption 3 holds. Then for all $k\geq 1$ the iterates of (\ref{increment1})--(\ref{increment3}) satisfy
  \end{proposition}
 \begin{eqnarray*}
 d(x_{k+1},\calX)^2
 \leq
  d(x_{k},\calX)^2
  - 2\alpha_k c d(x_k,\calX)^{\frac{1}{\theta}}
  +
  \alpha_k^2 m^2 G^2.
 \end{eqnarray*}
 
This is the same as the main recursion we analyze in (\ref{ABiggy}) with $G^2$ replaced by $m^2 G^2$. Thus all our results in the following sections apply to the incremental subgradient method (\ref{increment1})--(\ref{increment3}) with this change in constants. 

 \subsection{Goffin's Condition Number}\label{goff_cond}  
  Goffin \cite{goffin1977convergence} discussed a condition number for quantifying the convergence rate of subgradient methods. The condition number is a generalization of the ordinary notion defined for a smooth strongly convex function as the ratio of the Lipschitz constant of the gradient to the strong convexity parameter. In contrast Goffin's condition number requires neither smoothness or strong convexity. The condition number is also more general than Shor's eccentricity measure \cite{shor2012minimization}. The condition number for a convex function $h$ is defined as
  \begin{eqnarray}\label{goffinCond}
  \mu_{h}=\inf\left\{\frac{\langle u ,x-x^*_p\rangle}{\|u\|\|x-x_p^*\|}: x\in\calC\backslash \calX_h,u\in\partial h(x),x_p^*=\proj_{\calX_h}(x)\right\}.
  \end{eqnarray}
 By convexity and the Cauchy-Schwarz inequality $0\leq\mu_h\leq 1$. Goffin showed that if $h$ satisfies HEB$(c,\theta)$ with $\theta=1$ and $\|g\|\leq G$ for all $g\in\partial h(x),x\in\calC$, then it satisfies (\ref{goffinCond}) with
  \begin{eqnarray*}
  \mu_h\geq\frac{c}{G}=\frac{1}{\kappa}
  \end{eqnarray*}
  which proves that functions satisfying (\ref{goffinCond}) with $\mu_h>0$ are more general than weakly sharp functions. 
  
  Our results for $\theta=1$ throughout this manuscript can be extended to functions satisfying (\ref{goffinCond}) with $\mu_h>0$ if we make a slight modification to the subgradient method.
  \begin{lemma}[\cite{goffin1977convergence}]
  Let $\{x_k\}$ be a sequence satisfying
  \begin{eqnarray}\label{normalizedSG}
  x_{k+1}=P_\calC\left(x_k-\alpha_k\frac{g_k}{\|g_k\|}\right):\forall k\geq 1,g_k\in\partial h(x_k),x_1\in\calC.
  \end{eqnarray}
  If $\calX_h$ is nonempty and $h$ is convex, closed, and proper (CCP) and satisfies (\ref{goffinCond}) with $\mu_h>0$, then for all $k\geq 1$
  \begin{eqnarray*}
  d(x_{k+1},\calX_h)^2\leq d(x_k,\calX_h)^2-2\alpha_k\mu_h d(x_k,\calX_h)+\alpha_k^2.
  \end{eqnarray*}\label{lemGoff}
  \end{lemma}  
  This is the same recursion as (\ref{ABiggy}) with $G=1$, $\theta=1$,  and $c=\mu_h$. Thus all the results derived in this manuscript for HEB with $\theta=1$ can be derived for the scheme (\ref{normalizedSG}) applied to functions satisfying (\ref{goffinCond}) so long as $c$ is replaced by $\mu_h$ and $G=1$. Also note that Lemma \ref{lemGoff} does not require that the subgradients are uniformly bounded over $\calC$.

\section{Proof of Theorems \ref{thmDimSum}, \ref{ThmLargeTheta}, and \ref{thmD1}}\label{secProofDecay}
\subsection{Preliminaries}\label{secProofDecay1}
In order to determine the convergence rate of the recursion (\ref{ABiggy}) derived in Prop. \ref{Prop_keyRecur} under generic nonsummable stepsizes, we need two Lemmas. We start with a result from \cite{PolyakIntro} which considers (\ref{ABiggy}) when $\theta<\frac{1}{2}$ without the nuisance term $\alpha_k^2 G^2$. 
\begin{lemma}
Suppose
$$
0\leq u_{k+1}\leq u_k - \gamma_k u_k^{1+q}
$$
for $k=0,1,\ldots$ where $\gamma_k\geq 0$ and $q>0$. Then 
$$
u_k
\leq
u_0\left(1+qu_0^q\sum_{i=0}^{k-1}\gamma_i\right)^{-\frac{1}{q}}.
$$\label{PolyakLemma2}
\end{lemma} 
\begin{proof}
\cite[Lemma 6 pp. 46]{PolyakIntro}.
\end{proof} 
We will also use the following estimates for the sum of stepsizes $\sum_{i=k_0}^k\alpha_i$. 
\begin{lemma}\label{sumSteps}
Let $k\geq k_0\geq 1$.
\begin{enumerate}
\item If $p\in(0,1)$
\begin{eqnarray*}
\sum_{i=k_0}^k i^{-p}\geq \frac{(k+1)^{1-p}-k_0^{1-p}}{1-p}.
\end{eqnarray*}
\item 
If $p=1$ 
\begin{eqnarray*}
\sum_{i=k_0}^ki^{-p}\geq \ln\frac{k+1}{k_0}.
\end{eqnarray*}
\end{enumerate} 	
\end{lemma}
\begin{proof}
A straightforward integral test.
\end{proof}

\subsection{Main Proof for Theorems \ref{thmDimSum} and \ref{ThmLargeTheta}}
Continuing with the main analysis, the goal is to derive convergence rates for a sequence $e_k$ satisfying (\ref{ABiggy}). To this end, let
\begin{eqnarray}\label{defI}
I=\{k: \alpha_k G^2\geq  c e_k^\gamma\}.
\end{eqnarray}
Recall the notation $\gamma=1/(2\theta)$. 
We will consider three types of iterates and bound the convergence rate in each case. First, for those iterates $k\in I$ it is easy to derive the convergence rate. Second, we will bound the rate for an iterate in $I^c$ when the previous iterate is in $I$. Finally we will consider $s$ consecutive iterates in $I^c$, for which we can use the inequality in (\ref{defI}) to simplify recursion (\ref{ABiggy}). Note that $s$ can be arbitrarily large. In particular when $I$ is finite there are an unbounded number of consecutive iterates in $I^c$. Together these three cases cover all possible iterates. 

First for, $k\in I$ and $\alpha_k>0$
\begin{eqnarray*}
	\alpha_k c e_k^\gamma \leq \alpha_k^2 G^2\implies e_k\leq \left(\frac{\alpha_k G^2}{c}\right)^{\frac{1}{\gamma}}.
\end{eqnarray*}
Thus the rate of $e_k$ is $O\left(\alpha_k^{\frac{1}{\gamma}}\right)$ for $k\in I$. In particular since $\alpha_k=\alpha_1k^{-p}$, then for $k\in I$ and $\alpha_1>0$
\begin{eqnarray}\label{firstCase}
e_k\leq \left(\frac{\alpha_1 G^2}{c}\right)^{2\theta} k^{-2 p\theta}.
\end{eqnarray}

Now assume $k\in I$ and $k+1\in I^c$. Then
\begin{eqnarray}\label{z}
	e_{k+1}\leq e_k + \alpha_k^2 G^2 \leq \left(\frac{\alpha_k G^2}{c}\right)^{\frac{1}{\gamma}}+ \alpha_k^2 G^2.
\end{eqnarray}
Now since $\frac{1}{\gamma}=2\theta\in(0,2)$, for $k\geq 1$
\begin{eqnarray*}
k^{-2p\theta}\geq k^{-2p}.
\end{eqnarray*}
Therefore (\ref{z}) implies that for $k\in I$, $k+1\in I^c$, and $k\geq 1$,
\begin{eqnarray}\label{case3}\label{case33}
e_{k+1}
&\leq&
C_1(k+1)^{-2p\theta}
\end{eqnarray}
where
\begin{eqnarray*}
C_1 = 
2^{2p\theta}
\left(
\left(\frac{\alpha_1 G^2}{c}\right)^{\frac{1}{\gamma}}
+
\alpha_1^2 G^2
\right).
\end{eqnarray*}

Next assume $k\in I$, $k+1\in I^c$, and $k+i\in I^c$ for $i=2,\ldots s$ for some $s\geq 2$. Then for $i=2,\ldots s$
\begin{eqnarray}\label{bigrecursion}
e_{k+i}< e_{k+i-1}-\alpha_kc e_{k+i-1}^\gamma.
\end{eqnarray}
To analyze the recursion (\ref{bigrecursion}) we consider $\theta<\frac{1}{2}$ and $\theta\geq \frac{1}{2}$ separately.

\noindent 
{\bf \underline{Case 1}: $\boldsymbol{\theta<\frac{1}{2}}$.} 

\noindent
 Now since $\gamma> 1$ we can apply Lemma \ref{PolyakLemma2} along with Lemma \ref{sumSteps} to (\ref{bigrecursion}) and derive for $i=2,\ldots,s$
\begin{eqnarray}
e_{k+i}
&\leq& 
e_{k+1}
\left[
1+\frac{1-2\theta}{2\theta}e_{k+1}^{\frac{1-2\theta}{2\theta}}\sum_{j=1}^{i-1}\alpha_{k+j}
\right]^{\frac{2\theta}{2\theta-1}}
\nonumber\\\label{flow}
&\leq&
e_{k+1}
\left[
1+\frac{\alpha_1(1-2\theta)}{2\theta(1-p)}e_{k+1}^{\frac{1-2\theta}{2\theta}}
\left(
(k+i)^{1-p}
-(k+1)^{1-p}
\right)
\right]^{\frac{2\theta}{2\theta-1}}.
\end{eqnarray}
Now consider the condition given in (\ref{cc}).
Note that since $p$ satisfies (\ref{dcond}), if (\ref{cc}) holds for $k=k_0$, it holds for all $k>k_0$. In particular if it holds for $k=0$, then it holds for all $k$. Continuing, 
if (\ref{cc}) holds then for all $k>k_0$
\begin{eqnarray}
	1-\frac{\alpha_1(1-2\theta)}{2\theta(1-p)}e_{k+1}^{\frac{1-2\theta}{2\theta}}(k+1)^{1-p}\geq 0\label{gg}
\end{eqnarray}
where we have used the fact that $k+1\in I^c$. Therefore since (\ref{gg}) holds we can simplify (\ref{flow}) to say that for $k\in I$ and $k+i\in I^c$ for $i=2,3,\ldots,s$, and $k>k_0$,
\begin{eqnarray}
e_{k+i}&\leq& e_{k+1}
\left[\frac{\alpha_1(1-2\theta)}{2\theta(1-p)}e_{k+1}^{\frac{1-2\theta}{2\theta}}
(k+i)^{1-p}
\right]^{\frac{2\theta}{2\theta-1}}
\nonumber\\
&\leq&
\left(\frac{\alpha_1(1-2\theta)}{2\theta(1-p)}\right)^{\frac{2\theta}{2\theta-1}}
(k+i)^{\frac{2\theta(1-p)}{2\theta-1}}
\label{late}.
\end{eqnarray}
The final case to consider is when $i=1,2,\ldots,s$ are in $I^c$. In this case, the same bound (\ref{flow}) can be derived but with $e_{1}$ replacing $e_{k+1}$. Thus for $i=2,3,\ldots s$ in $I$
\begin{eqnarray}
e_{i}
\leq
e_1
\left[1+\frac{\alpha_1(1-2\theta)}{2\theta(1-p)}e_{1}^{\frac{1-2\theta}{2\theta}}
\left(
i^{1-p}-1
\right)
\right]^{\frac{2\theta}{2\theta-1}}\label{firstI}.
\end{eqnarray}
Thus if $\alpha_1$ is chosen to satisfy (\ref{cc2}) then
\begin{eqnarray}
e_{i}\label{prettyPleaseFinal}
\leq
\left(\frac{\alpha_1(1-2\theta)}{2\theta(1-p)}\right)^{\frac{2\theta}{2\theta-1}}
i^{\frac{2\theta(1-p)}{2\theta-1}}.
\end{eqnarray} 
Combining (\ref{firstCase}), (\ref{case3}), (\ref{late}), and (\ref{prettyPleaseFinal}) establishes (\ref{lowTheta1}) and concludes the proof of Theorem \ref{thmDimSum}.

\noindent 
{\bf\underline{Case 2}: $\boldsymbol{\theta\geq\frac{1}{2}}$}

\noindent
Next we consider the case where $\frac{1}{2}\leq\theta\leq 1$ which will finish the proof of Theorem \ref{ThmLargeTheta}. Before commencing we introduce the following Lemma which allows us to bound a decaying exponential by an appropriately scaled decaying polynomial of any degree.
\begin{lemma}\label{expToPoly}
Suppose $\delta>0$, then if $C_\delta\geq e^{-\delta}\delta^\delta$,
\begin{eqnarray}\label{ap}
\exp(-x)\leq C_\delta x^{-\delta}\quad\forall x>0.
\end{eqnarray}

\end{lemma}
\begin{proof}
Taking logs of both sides of (\ref{ap}) yields
\begin{eqnarray*}
-x\leq -\delta\ln x + \beta_\delta\quad\forall x>0
\end{eqnarray*}
where $\beta_\delta=\ln C_\delta$. Therefore
\begin{eqnarray*}
\beta_\delta
&\geq& \delta\ln x - x\quad\forall x>0
\end{eqnarray*}
which implies
\begin{eqnarray*}
\beta_\delta
&\geq& 
\max_{x>0}\{\delta\ln x - x\}.
\end{eqnarray*}
The right hand side is a smooth concave coercive maximization problem which therefore has a unique solution given by $x^*=\delta$. Hence
\begin{eqnarray*}
\beta_{\delta}\geq \delta\ln\delta - \delta
\end{eqnarray*}
which implies the Lemma. 
\end{proof}

Continuing, we consider $k\in I$, $k+1\in I^c$, and $k+i\in I^c$ for $i=2\ldots,s$ in the case where $\theta\geq\frac{1}{2}$, so $\gamma\leq 1$. Then since $k+i\in I^c$ for $i=2,\ldots s$,
\begin{eqnarray*}
0\leq\frac{e_{k+i-1}}{e_{k+1}}\leq 1\implies \left(\frac{e_{k+i-1}}{e_{k+1}}\right)^{\gamma}\geq \frac{e_{k+i-1}}{e_{k+1}}
\implies
e_{k+i-1}^\gamma \geq e_{k+1}^{\gamma-1}e_{k+i-1}.
\end{eqnarray*} 

Thus for $k\in I$, $k+1\in I^c$, and $k+i\in I^c$ for $i=2,\ldots, s$ for some $s\geq 2$
\begin{eqnarray}
e_{k+i}&\leq& e_{k+i-1}-\alpha_{k+i-1} c e_{k+i-1}^\gamma
\nonumber\\\label{innovate}
&\leq& 
e_{k+i-1}-\alpha_{k+i-1} e_{k+1}^{\gamma-1} c  e_{k+i-1}.
\end{eqnarray}
Now taking logs and using $\log(1-x)\leq-x$, 
\begin{eqnarray*}
\ln e_{k+i}
&\leq&
\ln e_{k+i-1} + \ln(1-e_{k+1}^{\gamma-1} c \alpha_{k+i-1} )
\\
&\leq&
\ln e_{k+i-1} - e_{k+1}^{\gamma-1} c\alpha_{k+i-1} .
\end{eqnarray*}
Now summing and using Lemma \ref{sumSteps} 
\begin{eqnarray*}
\ln e_{k+i}
&\leq&
\ln e_{k+1}-\alpha_1  e_{k+1}^{\gamma-1}c\sum_{i=k+1}^{k+i-1}i^{-p}
\\
&\leq&
\ln e_{k+1}-\frac{\alpha_1 e_{k+1}^{\gamma-1}c }{1-p}\left(
(k+i)^{1-p}-(k+1)^{1-p}
\right).
\end{eqnarray*}
This leads to
\begin{eqnarray}
e_{k+i}
&\leq&
e_{k+1}
\exp
\left\{-\frac{\alpha_1 e_{k+1}^{\gamma-1} c}{1-p}\left((k+i)^{1-p}-(k+1)^{1-p}\right)
\right\}\label{expForm}
\\\nonumber
&=&
\exp
\left\{-\frac{\alpha_1 e_{k+1}^{\gamma-1} c(k+i)^{1-p}}{1-p}\left(1-\left(\frac{k+1}{k+i}\right)^{1-p}\right)
\right\}.
\end{eqnarray}
We further consider two possible cases. If $i\geq k$, then 
\begin{eqnarray*}
\frac{k+1}{k+i}
\leq
\frac{k+1}{k+k}
=
\frac{1}{2}+\frac{1}{2k}
\end{eqnarray*}
therefore by concavity of $t^{1-p}$
\begin{eqnarray*}
\left(
\frac{k+1}{k+i}
\right)^{1-p}
\leq
2^{p-1}\left[1+\frac{1-p}{k}\right].
\end{eqnarray*}
Take 
$
k>3
$
so that
\begin{eqnarray*}
\frac{2^{p-1}(1-p)}{k}\leq \frac{1-2^{p-1}}{2}.
\end{eqnarray*}
Hence
\begin{eqnarray*}
1-\left(\frac{k+1}{k+i}\right)^{1-p}
\geq 
1-2^{p-1}\left[1+\frac{1-p}{k}\right]
\geq
1-2^{p-1}-\frac{2^{p-1}(1-p)}{k}
\geq 
 \frac{1-2^{p-1}}{2}.
\end{eqnarray*}
Hence if $3< k\leq i$ then 
\begin{eqnarray}
e_{k+i}&\leq& e_{k+1}\exp\left(
-\frac{(1-2^{p-1})\alpha_1  e_{k+1}^{\gamma-1} c}{2(1-p)}(k+i)^{1-p}
\right).
\end{eqnarray}
Now by Lemma \ref{expToPoly} for any $\delta_1>0$,
\begin{eqnarray*}
&&\exp\left\{-\frac{\alpha_1(1-2^{p-1})c e_{k+1}^{\gamma-1}}{2(1-p)}(k+i)^{1-p}\right\}
\\
&\leq&
\delta_1^{\delta_1}e^{-\delta_1} e_{k+1}^{1+\delta_1(1-\gamma)}
\left(
\frac{\alpha_1(1-2^{p-1})c}{2(1-p)}(k+i)^{1-p}
\right)^{-\delta_1}.
\end{eqnarray*}
Therefore using (\ref{case3}) for any $k\leq i$ and $k>3$ 
\begin{eqnarray}
e_{k+i}&\leq& 
\delta_1^{\delta_1}
C_1^{1+\delta_1(1-\gamma)}
\left(\frac{\alpha_1(1-2^{p-1}) c e }{2(1-p)}\right)^{-\delta_1}
(k+i)^{-\delta_1(1-p)}.
\label{precase}
\end{eqnarray}
Taking $\delta_1=\frac{2 p\theta}{1-p}$ and simplifying (\ref{precase}) yields
\begin{eqnarray}
\label{somecase}
e_{k+i}&\leq& 
C_1^{\frac{1+2p(\theta-1)}{1-p}}
\left(
\frac{\alpha_1(1-2^{p-1}) c e}{4p\theta}
\right)^{-\frac{2 p\theta}{1-p}}
(k+i)^{-2 p\theta}.
\end{eqnarray}

Next consider $k\geq i>1$. Now
\begin{eqnarray}
(k+i)^{1-p} - (k+1)^{1-p}
&=&
(k+i)^{1-p}
\left(
1-
\left(
\frac{k+1}{k+i}
\right)^{1-p}
\right)
\nonumber\\
&=&
(k+i)^{1-p}
\left(
1-
\left(
1-\frac{i-1}{k+i}
\right)^{1-p}
\right)
\nonumber\\
&\geq&
(k+i)^{1-p}
\left(
1-
\left(1-
\frac{i-1}{2k}
\right)^{1-p}
\right)
\nonumber\\
&\geq&
\frac{
(1-p)(k+i)^{1-p} 
(i-1)
}{2k}
\label{conc}\\\nonumber
&\geq&
\frac{1-p}{2}
k^{-p}
(i-1)
\end{eqnarray}
where in (\ref{conc}) we used the concavity of $t^{1-p}$.
Thus substituting this into (\ref{expForm}) implies for $k\geq i$
\begin{eqnarray*}
e_{k+i}
\leq 
e_{k+1}
\exp\left(
\frac{-\alpha_1  e_{k+1}^{\gamma-1} c(i-1)}
{2k^p}
\right).
\end{eqnarray*}
Therefore for all $\delta_2\geq 0$
it follows Lemma \ref{expToPoly} that
\begin{eqnarray}
e_{k+i}
&\leq&
e_{k+1}\exp\left(
\frac{-\alpha_1  e_{k+1}^{\gamma-1}c(i-1)}
{2k^{p}}
\right)
\nonumber\\
&\leq&
\delta_2^{\delta_2} e_{k+1}
\left(
\frac{\alpha_1  e_{k+1}^{\gamma-1} c(i-1)e}
{2k^{p}}
\right)^{-\delta_2}
\nonumber\\\label{step}
&\leq&
 C_1^{1+\delta_2(1-\gamma)}
\left(\frac{4\delta_2}{c \alpha_1 e}\right)^{\delta_2}  
k^{-2p\theta(1+\delta_2(1-\gamma))} k^{p\delta_2}i^{-\delta_2}
\end{eqnarray}
where we used $e_{k+1}\leq C_1 k^{-2p\theta}$ and $(i-1)^{-\delta_2}\leq 2^{\delta_2} i^{-\delta_2}$.
Now if we choose 
\begin{eqnarray}\label{optDelta2}
\delta_2=2\theta
\end{eqnarray}
then (\ref{step}) implies
\begin{eqnarray}\label{steps}
e_{k+i}\leq C_4 i^{-2\theta}
\end{eqnarray}
where
\begin{eqnarray}
C_4 = \left(\frac{8\theta C_1}{c \alpha_1 e}\right)^{2\theta}.  
\end{eqnarray}
Thus combining $e_{k+i}\leq e_{k+1}\leq C_1k^{-2p\theta}$ and (\ref{steps}) implies that for $i\leq k$
\begin{eqnarray*}
e_{k+i}\leq \max\{C_1,C_4\}\min\{k^{-2p\theta},i^{-2\theta}\}.
\end{eqnarray*}
Now since $-2\theta<-2p\theta$,
\begin{eqnarray*}
e_{k+i}\leq \max\{C_1,C_4\}\min\{k^{-2p\theta},i^{-2p\theta}\}\leq \frac{\max\{C_1,C_4\}}{\max\{k^{2p\theta},i^{2p\theta}\}}.
\end{eqnarray*}
If $2p\theta\geq 1$ then by convexity of $t^{2p\theta}$
\begin{eqnarray}
\max\{k^{2p\theta},i^{2p\theta}\}
\geq
\frac{1}{2}
\left(
k^{2p\theta}+i^{2p\theta}
\right)
\geq 
2^{-2p\theta}\left(k+i\right)^{2p\theta}.\label{aas}
\end{eqnarray}
On the other hand if $2p\theta<1$ then because $t^{2p\theta}$ is subadditive
\begin{eqnarray}\label{bees}
\max\{k^{2p\theta},i^{2p\theta}\}
\geq
\frac{1}{2}
\left(
k^{2p\theta}+i^{2p\theta}
\right)
\geq 
\frac{1}{2}\left(k+i\right)^{2p\theta}.
\end{eqnarray}
Combining (\ref{aas}) and (\ref{bees}) gives
\begin{eqnarray}\label{TheFinalCase}
e_{k+i}
\leq
4 \max\{C_1,C_4\}(k+i)^{-2p\theta}.
\end{eqnarray}
Finally we consider the case where the first $s$ iterates belong to $I^c$. Therefore, using (\ref{expForm}), for $i=1,2,\ldots, s$
\begin{eqnarray}
e_{i}
&\leq&
e_{1}
\exp
\left\{-\frac{\alpha_1 e_{1}^{\gamma-1} c}{1-p}\left(i^{1-p}-1\right)
\right\}.\label{firstI2}
\end{eqnarray}
Now since for $x\geq 1$, $x-1\geq \frac{x}{2}$, this implies that
\begin{eqnarray*}
e_{i}
&\leq&
e_{1}
\exp
\left\{-\frac{\alpha_1 e_{1}^{\gamma-1} c}{2(1-p)}i^{1-p}
\right\}.
\end{eqnarray*}
Using Lemma \ref{expToPoly} this implies that for any $\delta_3>0$
\begin{eqnarray}
e_{i}\label{cases}
&\leq&
e^{-\delta_3}\delta_3^{\delta_3}
e_{1}
\left(\frac{\alpha_1 e_{1}^{\gamma-1} c}{2(1-p)}i^{1-p}
\right)^{-\delta_3}.
\end{eqnarray}
and we will use $\delta_3=\frac{2p\theta}{1-p}$.

Combining (\ref{firstCase}), (\ref{case3}), (\ref{somecase}), (\ref{TheFinalCase}), and (\ref{cases}) yields the desired result (\ref{BigThetaResult}) and concludes the proof of Theorem \ref{ThmLargeTheta}.
 
\subsection{Proof of Theorem \ref{thmD1}}

The format of the proof is identical to Theorems \ref{thmDimSum} and \ref{ThmLargeTheta}. As before it is based on the set $I$ defined in (\ref{defI}) and we consider three types of iterates. First we bound the convergence rate for iterates in $I$, second for iterates in $I^c$ when the previous iterate is in $I$. And finally for $s$ consecutive iterates in $I^c$ where $s$ may be unbounded. 

If $k\in I$ then repeating (\ref{case3}) yields 
\begin{eqnarray}\label{newcase}
e_k\leq \frac{\alpha_1 G^2}{c} k^{-1}.
\end{eqnarray}
Similarly for $k\in I$ and $k+1\in I^c$, 
\begin{eqnarray}\label{newcase2}
e_{k+1}\leq \frac{2\alpha_1 G^2}{c} (k+1)^{-1}.
\end{eqnarray}
Finally for $k\in I$, $k+1\in I^c$, and $k+i\in I^c$, for $i=2,\ldots, s$, then repeating (\ref{innovate}) but with $\gamma=1$ this time,
\begin{eqnarray*}
e_{k+i}\leq e_{k+i-1}(1-c \alpha_{k+i-1}).
\end{eqnarray*}
Taking logs, using $\log(1-x)\leq -x$ and summing yields
\begin{eqnarray*}
	\log e_{k+i}
	&\leq& 
	\log e_{k+1} - c \alpha_1 \sum_{j=k+1}^{k+i-1} j^{-1}
	\\
	&\leq&
	\log e_{k+1}-c \alpha_1 \left(\log(k+i)-\log(k+1)\right)
\end{eqnarray*}
where we applied Lemma \ref{sumSteps} in the second inequality. This yields for all $k\in I$ and $k+i\in I^c$ for $i=2,3,\ldots, s$ for some $s\in \mathbb{N}$
\begin{eqnarray}
e_{k+i}&\leq& e_{k+1}\left(\frac{k+i}{k+1}\right)^{-c \alpha_1 }.
\label{need}
\end{eqnarray}
Using (\ref{newcase2}) yields 
\begin{eqnarray}
e_{k+i}&\leq&\frac{2 \alpha_1 G^2}{c} (k+1)^{-1}(k+1)^{c \alpha_1 }
(k+i)^{-c \alpha_1}
\nonumber\\
&\leq&
\frac{2 \alpha_1 G^2}{c}
(k+i)^{-c \alpha_1}
\label{thisCase}
\end{eqnarray}

Finally we consider the case where the initial iterates $i=1,2,\ldots,s$ are in $I^c$. Therefore repeating (\ref{need}) with $k=0$ gives
\begin{eqnarray}\label{fin}
e_{i}&\leq& e_{1} i^{-c \alpha_1 }
\end{eqnarray}

Combining (\ref{newcase}), (\ref{newcase2}), (\ref{thisCase}), and (\ref{fin}) yields (\ref{ThmD1Result}) and concludes the proof of Theorem \ref{thmD1}.  

\subsection{Proof of Proposition \ref{PropThta05}}\label{sec:PropProof}
As previously mentioned, this argument is a direct extension of \cite[Thm. 4]{karimi2016linear}. For $\theta=\frac{1}{2}$, (\ref{ABiggy}) reads as 
\begin{eqnarray*}
e_{k+1}\leq (1-2\alpha_k c)e_k+\alpha_k^2 G^2.
\end{eqnarray*}
We consider the choice $\alpha_k =\frac{2k+1}{2c(k+1)^2}$. Then
\begin{eqnarray*}
e_{k+1}
\leq
\left(
1-\frac{2k+1}{(k+1)^2}
\right)e_k+
\frac{G^2(2k+1)^2}{4 c^2(k+1)^4}.
\end{eqnarray*}
Multiplying both sides by $(k+1)^2$ yields
\begin{eqnarray*}
(k+1)^2e_{k+1}
&\leq&
k^2 e_k
+
\frac{G^2(2k+1)^2}{4 c^2(k+1)^2}
\\
&\leq&
k^2 e_k
+
\frac{G^2}{c^2}
\\
&\leq&
e_1+\frac{G^2}{c^2}k.
\end{eqnarray*}
Therefore
\begin{eqnarray*}
e_{k+1}
\leq
\frac{e_1}{(k+1)^2}
+
\frac{G^2}{c^2(k+1)}.
\end{eqnarray*}

{\bf Acknowledgments.} We thank Prof. Niao He for many illuminating and important discussions. 


\bibliographystyle{spmpsci}
\bibliography{refs}

\end{document}